\documentclass[12pt,leqno]{article}
\usepackage{amssymb,amsmath,amsthm}
\usepackage[all]{xy}
\usepackage{graphics}

\usepackage{enumerate}
\usepackage{mathrsfs}
\usepackage{color}
\usepackage[colorlinks=true, pdfstartview=FitV, linkcolor=blue,%
citecolor=blue, urlcolor=blue]{hyperref}

\newenvironment{bleu}
{\relax\color{blue}}
{\hspace*{.3ex}\relax}
\newcommand{\beb}{\begin{bleu}}
\newcommand{\eb}{\end{bleu}}

\newcommand{\nc}{\newcommand}
\nc{\on}{\operatorname}

\newlength{\my}
\nc{\noi}{\noindent}

\renewcommand{\Re}{\operatorname{Re}}
\renewcommand{\Im}{\operatorname{Im}}

\newtheorem{theorem}{Theorem}[section]
\newtheorem{proposition}[theorem]{Proposition}
\newtheorem{lemma}[theorem]{Lemma}
\newtheorem{corollary}[theorem]{Corollary}

\theoremstyle{definition}

\newtheorem{definition}[theorem]{Definition}
\newtheorem{notation}[theorem]{Notation}
\newtheorem{example}[theorem]{Example}

\newtheorem{remark}[theorem]{Remark}

\newtheorem{conjecture}[theorem]{Conjecture}

\nc{\Rem}{\begin{remark}}
\nc{\enrem}{\end{remark}}
\nc{\Conj}{\begin{conjecture}}
\nc{\enconj}{\end{conjecture}}

\nc{\RR}{\mathrm{R}}
\nc{\LL}{\mathrm{L}}

\newcommand{\C}{{\mathbb{C}}}

\newcommand{\R}{{\mathbb{R}}}
\newcommand{\Q}{{\mathbb{Q}}}

\newcommand{\BBP}{{\mathbb P}}

\newcommand{\BBD}{{\mathbb D}}

\newcommand{\BBH}{{\mathbb H}}
\newcommand{\BBHd}{{\mathbb H}^*}

\def\phi{{\varphi}}
\def\epsilon{\varepsilon}

\newcommand{\cor}{{\bf k}}

\newcommand{\icor}{\rm I{\bf k}}
\newcommand{\iC}{\rm I{\C}}
\newcommand{\iD}{\rm I{\D}}

\newcommand{\isha}{\rm I{\sha}}

\newcommand{\iDrm}{\rm I{\Drm}}

\def\D{\mathscr{D}}
\def\O{\mathscr{O}}
\def\sha{\mathscr{A}}
\def\shb{\mathscr{B}}
\def\shc{\mathscr{C}}
\def\shd{\mathscr{D}}
\def\she{\mathscr{E}}
\def\shf{\mathscr{F}}

\def\shi{\mathscr{I}}

\def\shk{\mathscr{K}}
\def\shl{\mathscr{L}}
\def\shm{\mathscr{M}}
\def\shn{\mathscr{N}}
\def\sho{\mathscr{O}}

\def\sht{\mathscr{T}}

\newcommand{\rmpt}{{\rm pt}}

\newcommand{\into}{\hookrightarrow}

\renewcommand{\to}[1][]{\xrightarrow[]{#1}}
\newcommand{\from}[1][]{\xleftarrow[]{#1}}
\newcommand{\isoto}[1][]{\xrightarrow[#1]%
{{\raisebox{-.6ex}[0ex][-.6ex]{$\mspace{1mu}\sim\mspace{2mu}$}}}}
\newcommand{\isofrom}[1][]{\xleftarrow[#1]%
{{\raisebox{-.6ex}[0ex][-.6ex]{$\mspace{1mu}\sim\mspace{2mu}$}}}}

\newcommand{\To}[1][\rule{1ex}{0pt}]{\xrightarrow{\hs{.6ex}#1\hs{.6ex}}}

\newcommand{\muHom}[1][]{\mathrm{Hom}^\mu_{\raise1.5ex\hbox to.1em{}#1}}
\newcommand{\Hom}[1][]{\mathrm{Hom}_{\raise1.5ex\hbox to.1em{}#1}}
\newcommand{\RHom}[1][]{\RR\mathrm{Hom}_{\raise1.5ex\hbox to.1em{}#1}}
\newcommand{\Ext}[2][]{\mathrm{Ext}_{\raise1.5ex\hbox to.1em{}#1}^{#2}}
\renewcommand{\hom}[1][]{{\mathscr{H}\mspace{-4mu}om}_{\raise1.5ex\hbox to.1em{}#1}}
\newcommand{\rhom}[1][]{{\RR\mathscr{H}\mspace{-3mu}om}_{\raise1.5ex\hbox to.1em{}#1}}
\newcommand{\rhomc}[1][]
{{\mathscr{H}\mspace{-3mu}om}^*_{\raise1.5ex\hbox to.1em{}#1}}

\newcommand{\cihom}[1][]
{{\mathscr{I}\mspace{-3mu}}{hom}^+_{\raise1.5ex\hbox to.1em{}#1}}

\nc{\ihom}[1][]{{\shi\mspace{-3mu}hom}_{\raise1.5ex\hbox to.1em{}#1}}
\nc{\rihom}[1][]{{\mspace{2mu}\mathrm{R}\shi\mspace{-3mu}hom}_{\raise1.5ex\hbox to.1em{}#1}}
\nc{\fihom}[1][]{{\shi\mspace{-3mu}hom}^{\mathrm{E}}_{\raise1.5ex\hbox to.1em{}#1}}
\nc{\FHom}[1][]{{\mathrm{RHom}^{\mathrm{E}}_{\raise1.5ex\hbox to.1em{}#1}}}
\nc{\fhom}[1][]{{\mathscr{H}%
\mspace{-3mu}om}^{\mathrm{E}}_{\raise1.5ex\hbox to.1em{}#1}}

\nc{\Tam}{{\rm E}}

\newcommand{\ext}[2][]{{\mathscr{E}xt}_{\raise1.5ex\hbox to.1em{}#1}^{#2}}
\newcommand{\Tor}[2][]{\mathrm{Tor}^{\raise1.5ex\hbox to.1em{}#1}_{#2}}
\newcommand{\tens}[1][]{\mathbin{\otimes_{\raise1.5ex\hbox to-.1em{}{#1}}}}
\newcommand{\ltens}[1][]{\mathbin{\overset{\mathrm{L}}\tens}_{#1}}

\newcommand{\lltens}[1][]{{\mathop{\tens}\limits^{\rm L}}_{#1}}

\newcommand{\etens}{\mathbin{\boxtimes}}

\newcommand{\shend}{\operatorname{{\she\mspace{-2mu}\mathit{nd}}}}
\newcommand{\Endo}[1][]{\mathrm{End}_{\raise1.5ex\hbox to.1em{}#1}}

\newcommand{\Aut}[1][]{\mathrm{Aut}_{\raise1.5ex\hbox to.1em{}#1}}
\newcommand{\sect}{\Gamma}
\newcommand{\rsect}{\mathrm{R}\Gamma}

\newcommand{\conv}[1][]{\mathop{\circ}\limits_{#1}}
\newcommand{\ctens}[1][]{\mathbin{\overset{+}\tens}_{#1}}

\newcommand{\cconv}[1][]{\mathop{\circ}\limits^{#1}}
\newcommand{\Dconv}{\cconv[{\mathrm D}]}

\newcommand{\econv}[1][]{\mathop{\circ}\limits^{\mathrm{E}}\limits_{#1}}

\newcommand{\cetens}[1][]{\mathop{\etens}\limits^{+}\limits_{#1}}

\newcommand{\VV}{{\mathsf{V}}}
\newcommand{\VVd}{{\mathsf{V}^*}}
\newcommand{\VVp}{\mathsf{V'}}
\newcommand{\VVdp}{\mathsf{V'}^*}

\newcommand{\WW}{{\mathbb{V}}}
\newcommand{\WWd}{{\mathbb{V}^*}}

\newcommand{\PP}{\mathsf{P}}
\newcommand{\Drm}{\mathsf{D}}
\newcommand{\OO}{\mathsf{O}}

\newcommand{\oim}[1]{{#1}_*}
\newcommand{\eim}[1]{{#1}_!}
\newcommand{\roim}[1]{\RR{#1}_*}
\newcommand{\reim}[1]{\RR{#1}_!}

\newcommand{\reeim}[1]{\RR{#1}_{\mspace{1mu}!!}}
\newcommand{\opb}[1]{#1^{-1}}

\newcommand{\epb}[1]{#1^{\,!}\,}

\newcommand{\Dtens}[1][]{\overset{\mathrm{D}}\otimes_{\raise1.5ex\hbox to-.1em{}#1}}
\newcommand{\Detens}[1][]{\overset{\mathrm{D}}\etens_{\raise1.5ex\hbox to-.1em{}#1}}
\newcommand{\Ddual}{{\BBD}}

\newcommand{\Dopb}[1]{{\mathrm{D}}{#1}^{*}}
\newcommand{\Doim}[1]{{\mathrm{D}}{#1}_{*}}

\newcommand{\good}{\mathrm{good}}
\newcommand{\qgood}{\mathrm{q\text-good}}
\newcommand{\ghol}{\mathrm{g\text-hol}}

\nc{\rE}{\mathrm{E}}
\nc{\enh}{\mathsf{E}}
\newcommand{\Toim}[1]{\enh{#1}_*}
\newcommand{\Teeim}[1]{\enh{#1}_{\mspace{1mu}!!}}
\newcommand{\Topb}[1]{\enh\mspace{2mu}#1^{-1}}
\newcommand{\Tepb}[1]{\enh\mspace{2mu}#1^{\,!}}

\nc{\EF}[1][]{{}^{\enh}{\mspace{-3mu}\shf_{#1}}}
\nc{\EFa}[1][]{{}^{\enh}{\mspace{-3mu}\shf^a_{#1}}}
\nc{\FS}[1][]{{}^{\rm S}{\mspace{-3mu}\shf_{#1}}}
\nc{\FSa}[1][]{{}^{\rm S}{\mspace{-3mu}\shf^a_{#1}}}
\nc{\Leg}[1][]{{\rm Conv}{(#1)}}
\nc{\dom}{\mathrm{dom}}
\nc{\domo}{\dom^\circ}

\newcommand{\md[1]}{{\rm Mod}({#1})}

\newcommand{\mdrc[1]}{{\rm Mod_{{\Rc}}}({#1})}
\newcommand{\mdrcc[1]}{{\rm Mod^c_{{\Rc}}}({#1})}

\nc{\sHH}{\mathscr{H}\mspace{-4mu}\mathscr{H}}

\nc{\sMH}{\mathscr{M}\mspace{-4mu}\mathscr{H}}

\newcommand{\eqdot}{\mathbin{:=}}
\newcommand{\seteq}{\mathbin{:=}}

\newcommand{\cl}{\colon}
\newcommand{\scbul}{{\,\raise.4ex\hbox{$\scriptscriptstyle\bullet$}\,}}

\newcommand{\tw}[1]{\widetilde{#1}}
\newcommand{\twX}{{\widetilde{X}}}
\newcommand{\rmH}{{\mathrm{H}}}
\newcommand{\rmE}{{\mathrm{E}}}
\newcommand{\rmd}{{\mathrm{d}}}

\newcommand{\olom}{\varpi}

\newcommand{\ol}{\overline}
\newcommand{\bl}{\bigl(}
\newcommand{\br}{\bigr)}
\newcommand{\ro}{{\rm(}}
\newcommand{\rf}{\,{\rm)}}

\newcommand{\rp}{{\rm)}}

\newcommand{\Rc}{{\R\text{-c}}}

\newcommand{\ba}{\begin{array}}
\newcommand{\ea}{\end{array}}

\nc{\be}{\begin{enumerate}}
\nc{\ee}{\end{enumerate}}
\newcommand{\bnum}{\begin{enumerate}[{\rm(i)}]}
\newcommand{\enum}{\end{enumerate}}
\newcommand{\banum}{\begin{enumerate}[{\rm(a)}]}
\newcommand{\eanum}{\end{enumerate}}

\newcommand{\eq}{\begin{eqnarray}}
\newcommand{\eneq}{\end{eqnarray}}
\newcommand{\eqn}{\begin{eqnarray*}}
\newcommand{\eneqn}{\end{eqnarray*}}

\newcommand{\set}[2]{\left\{#1 \mathbin{;} #2 \right\}}

\nc{\Proof}{\begin{proof}}
\nc{\QED}{\end{proof}}
\nc{\Prop}{\begin{proposition}}
\nc{\enprop}{\end{proposition}}

\def\rop{{\rm op}}
\def\op{{\rm op}}

\def\Op{{\rm Op}}
\def\dist{{\rm dist}}

\def\hol{{\rm hol}}

\DeclareMathOperator{\id}{id}

\DeclareMathOperator{\supp}{supp}
\DeclareMathOperator{\ori}{{or\mspace{2mu}}}
\DeclareMathOperator{\chv}{char}

\newcommand{\Der}[1][]{\mathsf{D}^{#1}}
\newcommand{\Derb}{\Der[\mathrm{b}]}

\newcommand{\RD}{\mathrm{D}}

\newcommand{\coh}{{\rm coh}}

\newcommand{\dT}{{\dot{T}}}

\newcommand{\II[1]}{{\rm I}({#1})}
\newcommand{\IIrc[1]}{{\rm I_{\R-c}}({#1})}

\newcommand{\indcc}{{\rm Ind}(\shc)}
\newcommand{\indc}{\rm IC}

\newcommand{\mdcp[1]}{{\rm Mod^{{\rm c}}}({#1})}
\newcommand{\BDC}{\Derb}
\newcommand{\TDC}{{\rm E^b}}

\newcommand{\mop}{\mathrm{r}}

\newcommand{\Tmp}{\mathsf{T}}
\newcommand{\OEn}{\O^{\mspace{2mu}\enh}}
\newcommand{\DbT}{\Db^\Tmp}
\newcommand{\DbE}{\Db^\enh}
\newcommand{\OvE}{\Omega^\enh}
\newcommand{\drE}{\mathcal{DR}^\enh}
\newcommand{\solE}{\mspace{1mu}\mathcal{S}ol^{\mspace{1mu}\enh}}

\nc{\wc}[1]{\overset{\mbox{$\scriptscriptstyle\vee$}}{#1}}
\nc{\field}{\cor}
\nc{\bM}{\widehat{M}}
\nc{\bN}{\widehat N}
\nc{\bX}{{\widehat X}}
\nc{\bY}{\widehat Y}
\nc{\bL}{\widehat L}
\nc{\baf}{\widehat f}
\nc{\bR}{{\ol\R}}
\nc{\bV}{{\ol \VV}}
\nc{\bW}{{\ol \WW}}

\nc{\bVd}{{\ol \VVd}}
\nc{\bWd}{{{\ol \WW}^*}}

\newcommand{\fM}{{M_\infty}}
\newcommand{\fN}{{N_\infty}}
\newcommand{\fS}{{S_\infty}}
\newcommand{\fX}{{X_\infty}}
\newcommand{\fY}{{Y_\infty}}

\newcommand{\fR}{{\R_\infty}}
\newcommand{\fV}{{\VV_\infty}}
\newcommand{\fVd}{{\VV^*_\infty}}
\newcommand{\fW}{{\WW_\infty}}
\newcommand{\fWd}{{\WW^*_\infty}}

\nc{\oM}{{\ol M}}
\nc{\oN}{\ol N}
\nc{\oX}{{\ol X}}
\nc{\oS}{\ol S}
\nc{\oY}{\ol Y}
\nc{\oL}{\ol L}
\nc{\oR}{{\ol\R}}
\nc{\Tl}{\mathrm{L^E}}
\nc{\Tr}{\mathrm{R^E}}

\newcommand{\Msa}{{M_{\rm sa}}}
\newcommand{\Db}{{\cal D} b}
\newcommand{\Dbt}{{\cal D} b^{\mathrm t}}
\newcommand{\Cinft}[1][X]{\mathcal{C}^{\infty,\mathrm t}_{#1}}
\newcommand{\thom}{Thom}

\newcommand{\Ot}[1][X]{\mathcal{O}^{\mspace{2mu}\mathrm t}_{#1}}

\newcommand{\Ovt}{\Omega^{\mspace{1.5mu}\mathrm t}}

\newcommand{\dr}{\mathcal{DR}}
\newcommand{\drt}{\mathcal{DR}^{\mathrm t}}

\newcommand{\sol}{\mathcal Sol}

\newcommand{\solt}{\mathcal Sol^{\mspace{2.5mu}\mathrm t}}
\newcommand{\reghol}{{\mathrm{rh}}}

\newcommand{\indlim}[1][]{\mathop{\varinjlim}\limits_{#1}}
\newcommand{\sindlim}[1][]{\smash{\mathop{\varinjlim}\limits_{#1}}\,}

\newcommand{\sprolim}[1][]{\smash{\mathop{\varprojlim}\limits_{#1}}\,}

\newcommand{\inddlim}[1][]{\mathop{\text{\rm``{$\varinjlim$}''}}\limits_{#1}}
\newcommand{\sinddlim}[1][]{\smash{\mathop{\text{\rm``{$\varinjlim$}''}}\limits_{#1}}\,}

\nc{\eps}{\varepsilon}
\nc{\hs}{\hspace*}
\nc{\nn}{\nonumber}
\nc{\tM}{\widetilde{M}}
\nc{\h}{\mathbf{h}}
\nc{\tf}{\tilde{f}}
\nc{\trf}{{{}^{\mathrm{t}}\mspace{-3mu}f}}
\nc{\codim}{\on{codim}}
\nc{\lh}{\mathscr{H}}
\nc{\bwr}{\mbox{\large{$\wr$}}}
\nc{\dTi}{\dT^{*,\mathrm{in}}}
\nc{\Cd}{\mathrm{C}}
\nc{\tK}{\widetilde{K}}
\nc{\aMM}{a_{M\times M}}
\nc{\e}{\mspace{1mu}\mathrm{e}\mspace{1mu}}
\nc{\lan}{\langle}
\nc{\ran}{\rangle}
\nc{\la}{\lambda}

\numberwithin{equation}{section}

\begin{document}

\title{Irregular holonomic kernels and Laplace transform}

\author{Masaki Kashiwara and Pierre Schapira}

\date{}

\maketitle

\begin{abstract}
Given a (not necessarily regular) holonomic $\D$-module $\shl$  defined on the product of two complex manifolds, we prove that the correspondence associated with $\shl$ commutes (in some sense) with the De Rham functor. We apply this result to the study of the classical Laplace transform. The main tools used here are the theory of ind-sheaves 
and its enhanced version. 
\end{abstract}

\footnote{M.~K.\ was partially supported by Grant-in-Aid for
Scientific Research (B) 22340005, Japan Society for the Promotion of
Science.}

\tableofcontents

\section{Introduction}
Perhaps the most popular integral transform in Mathematics is the Fourier transform, or its complex version, the Laplace transform. It interchanges objects living on a finite-dimensional vector space $\VV$ with objects living  on  
the dual space $\VVd$. 
The kernel of this transform is $\e^{\langle x,y\rangle}$ and all the subtlety and difficulty of this transform come from the fact   that 
the $\D$-module on $\VV\times\VVd$ generated by 
this kernel is holonomic but is not regular.  In this paper,  we shall give tools to treat integral transforms associated with general holonomic kernels and apply them to the particular case of the Laplace transform.

Let us be more precise. For a complex manifold $(X,\sho_X)$ we denote by $\Omega_X$ the sheaf of differential forms of top degree, by $\D_X$ the sheaf of (finite-order) differential operators and by $d_X$ the complex dimension of $X$. We use the usual six operations for sheaves and denote by  $\Doim{g}$, $\Dopb{f}$ and  $\Dtens$  the operations of direct image, inverse image and tensor product   for $\D$-modules. Here, a sheaf or a 
$\D$-module should be understood in the derived sense, that is, in the bounded derived category of sheaves or 
$\D$-modules.

All along this paper, we shall use the language of ind-sheaves of~\cite{KS01} and the six operations for ind-sheaves, 
$\ltens$, $\rihom$, $\opb{f}$, $\epb{f}$, $\roim{f}$ and $\reeim{f}$. The main object of interest will be 
the ind-sheaf $\Ot[X]$  of holomorphic functions with tempered growth,  realized as the Dolbeault complex of the ind-sheaf $\Cinft[X]$ of $\mathrm{C}^\infty$-functions with tempered growth. On an open subanalytic subset $U$ the sections of this last sheaf are functions which have  polynomial growth at the boundary, as well as all their derivatives. The history of the ind-sheaf $\Ot[X]$  is  closely related to the solution of the Riemann-Hilbert problem for regular holonomic modules of~\cite{Ka80,Ka84}. Recall that the main tool 
 to solve this problem was the functor $\thom$, which in the language of ind-sheaves reads as 
$\rhom(\scbul,\Ot[X])$.  

We shall use the Sol and tempered Sol functors and the De Rham and tempered De Rham functors for $\D_X$-modules.
Denoting by  $\Ovt_X$ the ind-sheaf of tempered differential forms of top degree,  these functors are given by:
\begin{align*}
\sol_X(\shm)&=\rhom[\D_X](\shm,\sho_X),\quad &\solt_X(\shm)&=\rhom[\D_X](\shm,\Ot[X]),\\
\dr_X(\shm)&=\Omega_X\ltens[\D_X]\shm,\quad &\drt_X(\shm)&=\Ovt_X\ltens[\D_X]\shm.
\end{align*}

Consider a correspondence of complex manifolds:
\eq\label{diag:inttrans0}
&&\xymatrix@C=6ex@R=4ex{
&S\ar[ld]_-f\ar[rd]^-g&\\
{X}&&{Y}.
}\eneq
For a  $\D_S$-module $\shl$ and a  $\D_X$-module $\shm$ one sets:
\eqn
&&\shm\Dconv\shl\eqdot\Doim{g}(\Dopb{f}\shm\Dtens\shl).
\eneqn
For ind-sheaves  $L$ on $S$,  $F$ on $X$ and $G$ on $Y$ one sets
\eq
&&L\conv G\eqdot\reeim{f}(L\tens\opb{g}G),\quad \Psi_L(F)=\roim{g}\rihom(L,\epb{f}F).\label{eq:conv}
\eneq
For the notion of being quasi-good or good and the categories $\Derb_\good(\D_X)$ and  $\Derb_\qgood(\D_X)$,
see \S~\ref{subsectionDmod}.
\begin{theorem}\label{th:741200}
Let  $\shm\in\Derb_\qgood(\D_X)$ and let $\shl\in\Derb_\good(\D_{S})$.  Set $L\eqdot\sol_{S}(\shl)$ and assume
\bnum
\item
$\shl$ is  regular holonomic,
\item
$\opb{f}\supp(\shm)\cap\supp(\shl)$ is proper over $Y$.
\enum
Then  there is a natural isomorphism in $\Derb(\C_Y)$:
\eq\label{eq:functDRregcase}
&&\Psi_L(\drt_X(\shm))\,[d_X-d_S]\simeq\drt_Y(\shm\Dconv\shl).
\eneq
\end{theorem}
This result 
(which, to our knowledge, never appeared in the literature under this form) 
is an immediate  consequence of three deep results:\\
\newlength{\mylength}
\setlength{\mylength}{\textwidth}
\addtolength{\mylength}{-3ex}
(i)\ \parbox[t]{\mylength}{the  direct image functor $\Doim{g}$  commutes,
under an hypothesis of properness, with the tempered De Rham functor,}\\
(ii) the inverse image functor $\Dopb{f}$ commutes, up to a shift, with the tempered De Rham functor,\\
(ii) the formula, in which $\shn$ is regular  holonomic  and  $\shm$ is coherent on $X$:
\eq\label{eq:bjork00}
\rihom(\sol_X(\shn),\drt_X(\shm))&\simeq&\drt_X(\shm\Dtens\shn).
\eneq
As a corollary, one gets (under the same hypotheses) the adjunction formula  of~\cite{KS01}, in which $G$ is an ind-sheaf on $Y$:
\eq\label{eq:741200}
&&\ba{c}\RHom(L\conv G,\Ovt_{X}\ltens[\D_{X}]\shm)\,[d_X-d_S]
\simeq\RHom(G,\Ovt_{Y}\ltens[\D_{Y}](\shm\Dconv\shl)).
\ea\eneq
Note that a similar formula holds when replacing $\Ot[X]$ and $\Ot[Y]$  with their non tempered versions $\sho_X$ and $\sho_Y$ (and ind-sheaves with usual sheaves), but the hypotheses are different. Essentially, $\shm$ has to be coherent, $f$ non characteristic for $\shm$ 
and $\Dopb{f}\shm$ has to be transversal to the holonomic module $\shl$. On the other hand,  we do not need the regularity assumption on $\shl$. 
See~\cite{DS96} for such a non tempered formula (in a more particular setting). 

However, if one removes the hypothesis that the holonomic module $\shl$ is regular in Theorem~\ref{th:741200}, 
formulas~\eqref{eq:functDRregcase} and~\eqref{eq:741200} do not hold anymore  and we have to replace  $\Ot[X]$ with its enhanced version, the object 
$\OEn_X$ of~\cite{DK13}, and this is one of the main purpose of this paper. Let us briefly explain what is $\OEn_X$.

In order to keep in mind the behavior at infinity of our objects, one considers bordered spaces $\fM=(M,\bM)$, where 
$M$ is open in $\bM$. A typical example, which will play an essential role here, is the bordered space 
\eqn
&&\fR=(\R,\bR) \mbox{ where }\bR\eqdot\R\sqcup \{+\infty,-\infty\}.
\eneqn
For a commutative ring $\cor$ and a ``good'' topological space $M$, denote by 
$\Derb(\icor_M)$ the bounded derived category of ind-sheaves  of $\cor$-modules. 
Then one introduces  the quotient category
\eqn
&&\Derb(\icor_\fM)\eqdot\Derb(\icor_{\bM})/
\{F; \cor_{M}\tens F\simeq 0\}.
\eneqn
(Note that when working with usual sheaves, one would recover the category $\Derb(\cor_M)$ but the situation is different with ind-sheaves.) The six operations on ind-sheaves are easily extended to ind-sheaves on bordered spaces. 

Now consider the bordered space $M\times\fR=(M\times\R, M\times\bR)$ .
Denote by $\pi\cl M\times\fR\to M$ the projection. One defines the new category of enhanced ind-sheaves on $M$ by setting:
\eqn
\TDC(\icor_M)\eqdot\Derb(\icor_{M\times\fR})/\{F;\opb{\pi}\roim{\pi}F\isoto F\}.
\eneqn
The  quotient functor 
$\Derb(\icor_{M\times\fR})\to\TDC(\icor_M)$ admits a right and a left adjoint, denoted by $\Tr$ and $\Tl$, respectively. 
This category $\TDC(\icor_M)$ is closely related to constructions  initiated in Tamarkin~\cite{Ta08} (see also~\cite{GS12} for a detailed exposition and complements to Tamarkin's work). In particular it is endowed with a new tensor product, denoted by $\ctens$, and a new internal hom, denoted by $\cihom$.  The four operations for enhanced ind-sheaves associated with a morphism of manifolds $f$ are denoted by $\Toim{f}$, $\Teeim{f}$, $\Topb{f}$ and $\Tepb{f}$ and one also uses the bifunctor 
 $\FHom$ with values in $\Derb(\cor)$ (see Definition~\ref{def:fihom}). 
 These operations enjoy similar properties to the one for sheaves. 

Let $X$ be a complex manifold, $Y\subset X$ a 
 complex hypersurface and set $U=X\setminus Y$.
For $\varphi\in\O_X(*Y)$, one sets
\begin{align*}
\D_X \e^\varphi &= \D_X/\{P; P\e^\varphi=0 \text{ on } U\}, \quad
\she^{\,\varphi}_{U|X}=\D_X \e^\varphi(*Y).
\end{align*}
Then one introduces the object $\OEn_X$ of $\TDC(\iC_X)$ which plays a role analogous to the objects $\Ot[X]$ but contains more information. Denote by $i\colon X\times\fR \to X\times\BBP$ the natural morphism and denote by $\tau\in\C\subset \PP$ the affine variable in the complex projective line $\PP$. One sets:
\begin{align*}
\OEn_X 
&= \epb i \rhom[\D_\BBP](\she_{\C|\BBP}^\tau,\Ot[X\times\BBP])[2] \in \TDC(\iC_X), \\
\OvE_X &= \Omega_X \ltens[\O_X] \OEn_X.
\end{align*}
One defines the enhanced De Rham and Sol functors by 
\begin{align*}
\drE_{X}&\cl\Derb_\qgood(\D_{X})\to\TDC(\iC_X),\quad \shm\mapsto \OvE_{X}\ltens[\D_X]\shm,\\
\solE_{X}&\cl (\Derb_\qgood(\D_{X}))^\rop\to\TDC(\iC_X),\quad
\shm\mapsto\rhom[\D_X](\shm,\OEn_{X}).
\end{align*}
One also defines the enhanced version  $L\econv\scbul$ and $\Psi_L^\Tam(\scbul)$ of the functors 
$L\conv\scbul$ and $\Psi_L(\scbul)$  in~\eqref{eq:conv}.
 Then our main result is the following (see Theorem~\ref{th:7412}) which generalizes
Theorem~\ref{th:741200} to the case where $\shl$ is no more regular.

\begin{theorem}\label{th:741201}
Let $\shm\in\Derb_\qgood(\D_X)$, 
$\shl\in\Derb_\good(\D_{S})$ and set $L\eqdot\solE_{S}(\shl)$. 
 Assume that  $\shl$ is holonomic and that $\opb{f}\supp(\shm)\cap\supp(\shl)$ is proper over $Y$.
Then there is a natural isomorphism in $\TDC(\iC_Y)$:
\eqn
&&\Psi_L^\Tam(\drE_X(\shm))\,[d_X-d_S]\simeq\drE_Y(\shm\Dconv\shl).
\eneqn
\end{theorem}
Note that in the course of the proof, we shall need to strengthen the Riemann-Hilbert theorem of~\cite{DK13} and to prove the isomorphism for an arbitrary
holonomic $\D_X$-module $\shm$ (see~Theorem~\ref{th:bjork}):
\eqn\label{eq:bjorkmorf00}
&&\shm\Dtens\OEn_X\isoto\cihom(\solE_X(\shm),\OEn_X)\mbox{ in }\TDC(\iD_X).
\eneqn
The proof of this isomorphism, which follows from the same lines as in loc.\ cit., is  a  main technical part of this paper. 
As an easy application of our theorem, we obtain:

\begin{corollary}\label{cor:741201}
In the situation as in {\rm Theorem~\ref{th:741201}}, let $G\in \TDC(\iC_Y)$. Then there is a natural isomorphism in $\Derb(\C)$
\eqn
&&\FHom(L\econv G,\OvE_{X}\ltens[\D_{X}]\shm)\,[d_X-d_S]\\
&&\hspace{30ex}\simeq\FHom(G,\OvE_{Y}\ltens[\D_{Y}](\shm\Dconv\shl)).
\eneqn
\end{corollary}

Next, we shall apply these results to the Laplace transform. 
For that purpose, we need to treat first the Fourier-Sato transform, and its enhanced version. 
Let $\VV$ be a real finite-dimensional vector space of dimension $n$, $\VVd$ its dual. Recall that the Fourier-Sato transform,
denoted here by  $\FS[\VV]$,  is an equivalence of categories between conic sheaves on $\VV$ and conic sheaves on $\VV^*$. References are made to~\cite{KS90}. In~\cite{Ta08}, D.~Tamarkin has extended the Fourier-Sato transform to no more conic (usual) sheaves, by adding an extra variable. Here we generalize this last transform to enhanced ind-sheaves on
the bordered space $\fV=(\VV,\bV)$, where $\bV$ is the projective compactification of 
$\VV$. We introduced the kernels $L_\VV\eqdot\cor_{\{t=\langle x,y\rangle\}}$ 
and $L^a_\VV\eqdot\cor_{\{t=-\langle x,y\rangle\}}$ 
and define the enhanced Fourier-Sato functors 
\begin{align*}
&&\ba{c}
\EF[\VV]\cl \TDC(\icor_{\fV})\to\TDC(\icor_{\fVd}),\quad \EF[\VV](F)=F\econv L_\VV,\\
\EFa[\VV]\cl \TDC(\icor_{\fV})\to\TDC(\icor_{\fVd}),\quad \EFa[\VV](F)=F\econv L^a_\VV.
\ea
\end{align*}
We easily prove that the functors $\EF[\VV]$ and 
 $\EFa[\VVd]$ are equivalences of categories, quasi-inverse to each other up to shift (see Theorem~\ref{th:fourier}).
Moreover  the enhanced Fourier-Sato transform is compatible with the classical one, that is, we have a quasi-commutative  diagram of categories and functors (in which the vertical arrows are fully faithful functors):
\eq\label{diag:fourierefourier00}
&&\ba{c}\xymatrix{
\TDC(\icor_\fV)\ar[rr]^-{\EF[\VV]}&&\TDC(\icor_\fVd)\\
\Derb_{\R^+}(\cor_\VV)\ar[rr]^-{\FS[\VV]}\ar[u]_-{\epsilon_\VV}&&\Derb_{\R^+}(\cor_\VVd)\ar[u]_-{\epsilon_\VVd}.
}\ea
\eneq

Let now $\WW$ be a complex vector space of complex dimension $d_\WW$ and let $\WWd$ be its dual. 
We denote here by $\bW$ the projective compactification of $\WW$ and set $\fW=(\WW,\bW)$, $\BBH=\bW\setminus\WW$. 
We set for short $X=\bW\times\bWd$,  $U=\WW\times\WWd$
and consider the  Laplace kernel
\eq\label{eq:Laplaceker00}
&&\shl\eqdot\she^{\langle x,y\rangle}_{U|X}.
\eneq
Denote by $\Drm_\WW$ the Weyl algebra on $\WW$. 
As it is well-known, the Laplace kernel induces an isomorphism
$\Drm_\WW\simeq\Drm_\WWd$, hence an isomorphism
\eq\label{eq:Laplaceiso100}
&&\D_\bW(*\BBH)\Dconv\shl\simeq \D_\bWd(*\BBHd)\tens\det\WWd.
\eneq
Then our main theorem on the Laplace transform (see Theorem~\ref{th:laplace}) is:

\begin{theorem}\label{th:laplace00}
Isomorphism~\eqref{eq:Laplaceiso100} induces an isomorphism 
\eq\label{eq:Laplaceiso200}
&&\EF[\WW](\OEn_\fW)\simeq\OEn_\fWd\tens\det\WW\,[-d_\WW]\mbox{  in }\Derb((\iDrm_\WW)_\fWd).
\eneq
\end{theorem}
As an immediate application (see Corollary~\ref{cor:laplace}), we obtain
\begin{corollary}\label{cor:laplace00}
Isomorphism~\eqref{eq:Laplaceiso200} 
induces   an isomorphism in $\Derb(\Drm_\WW)$, 
functorial in $F\in\TDC(\iC_{\fW})${\rm:}
\eq
&&\FHom(F,\OEn_\fW)\simeq \FHom(\EF_\WW(F),\OEn_\fWd)\tens\det\WW\,[-d_\WW].
\eneq
\end{corollary}
When restricting this isomorphism to conic  sheaves on $\WW$, we recover 
the main theorem of~\cite{KS97} which asserts that 
for an $\R$-constructible and conic sheaf  $F\in\Derb(\C_\WW)$, the Laplace transform induces an isomorphism
\eq\label{eq:Laplaceiso97} 
\RHom(F,\Ot[\WW])\simeq \RHom(\FS[\WW](F),\Ot[\WWd])\tens\det\WW\,[-d_\WW].
\eneq
Several applications are given in loc.\ cit.\ and, by adding a variable, some non conic situations are also treated with the help of  this isomorphism in~\cite{Da12}. 

Here, as an application of Corollary~\ref{cor:laplace00}, we obtain the following result. 
For an open subset $U$ of $\WW$ subanalytic in $\bW$ and a continuous function $\phi$ on $U$ with subanalytic graph
in $\bW$, we can define the ind-sheaf  $\e^{\phi}\Dbt_M$. Roughly speaking,  it is the ind-sheaf of  distributions  $u$ such that   
 $\e^{-\phi} u$  is tempered. Consider the Dolbeault complex
\eq\label{eq:dolbeautDbt00} 
&&\e^\phi\Ot[\fW]\eqdot0\to \e^\phi{\Dbt_\bW}^{(0,0)}\to[\ol\partial]\cdots\to 
\e^\phi{\Dbt_\bW}^{(0,d_\WW)}\to0.
\eneq
Assume that $U$ is convex, $\phi$ is a convex function  and denote by $\phi^*$ its Legendre transform. Then,
 under some hypotheses, we prove that the complex $\e^\phi\Ot[\fW](U)$ is concentrated in degree $0$, the complex
 $\e^{-\phi^*}\Ot[\fW](\WWd)$ 
is concentrated in degree $d_\WW$ 
and the Laplace transformation interchanges these two complexes
(Corollary~\ref{cor:vanishOEn1}).
\begin{remark}
A detailed survey of the theory of regular and irregular holonomic D-modules will appear in~\cite{KS15}.
\end{remark}

\section{Enhanced ind-sheaves}

\subsection{Ind-sheaves}
In this subsection and the next one, we recall some results of~\cite{KS01}.

Let $M$ be a locally compact space countable at infinity and let $\cor$ be a commutative Noetherian ring with finite global dimension. (In this paper, all rings are unital.)
Recall that $\md[\cor_M]$ denotes the abelian category of sheaves of $\cor$-modules on $M$.
We denote by $\mdcp[\cor_M]$ the full subcategory consisting of sheaves with compact support. 
The category of ind-sheaves on $M$ is the category of ind-objects of $\mdcp[\cor_M]$. 
We set for short:
\eqn
&& \icor_M:= {\rm Ind}(\mdcp[\cor_M]), 
\eneqn
and  call an object of this category  an {\em ind-sheaf} on $M$. 
We denote by $\sinddlim$ the inductive limit in the category $\icor_M$.

The prestack $U\mapsto\II[\cor_U]$, $U$ open in $M$, is a stack.

We have two pairs $(\alpha_M,\iota_M)$ and $(\beta_M,\alpha_M)$  of adjoint functors 
\eqn
&& \xymatrix{
{\md[\cor_M]}\ar@<-1.5ex>[rr]_-{\beta_M}\ar@<1.5ex>[rr]^-{\iota_M}&&{\II[\cor_M]}\ar[ll]|-{\;\alpha_M\;}.
}
\eneqn
 If $F$ has compact support, $\iota_M(F)=F$ after identifying a category $\shc$ with a full subcategory of $\indcc$. More generally, $\iota_M(F)=\inddlim[U] F_U$ where $U$ ranges over the family of relatively compact open subsets of $M$.
 The functor $\alpha_M$ associates 
$\indlim[i] F_i$ ($F_i\in\mdcp[\cor_M]$, $i\in I$, $I$ small and filtrant) to the object $\inddlim[i] F_i$. If $\cor$ is a field, $\beta_M(F)$ coincides with the functor 
$\mdcp[\cor_M]\ni G\mapsto\sect(M;H^0(\RD_M'G)\tens F)$,
where  $\RD_M'$ is the duality functor $\rhom(\scbul,\cor_M)$.
\begin{itemize}
\item
$\iota_M$ is exact, fully faithful, and commutes with $\sprolim$, 
\item $\alpha_M$ is exact and commutes with $\sprolim$ and $\sindlim$,
\item $\beta_M$ is right exact, fully faithful and commutes with $\sindlim$,
\item $\alpha_M$ is left  adjoint to $\iota_M$,
\item $\alpha_M$ is right adjoint to $\beta_M$,
\item $\alpha_M\circ\iota_M\simeq \id_{\md[\cor_M]}$ 
and $\alpha_M\circ\beta_M\simeq \id_{\md[\cor_M]}$.
\end{itemize}
One denotes by $\ltens$ and $\rihom$ the (derived) operations of tensor product and  internal $\hom$. If $f\cl M\to N$ is a continuous map, one denotes by $\opb{f}$, $\epb{f}$, $\roim{f}$ and $\reeim{f}$ the (derived) operations of inverse and direct images. One also sets
\eqn
&&\rhom=\alpha_M\circ\rihom\cl \Derb(\icor_M)^\rop\times\Derb(\icor_M)\to\Derb(\cor_M).
\eneqn

\subsection{Subanalytic topology}
Here again, we recall some results of~\cite{KS01}.

Assume  that $M$ is a real analytic manifold. 
Denote by $\Op_M$ the category of its open subsets (the morphisms being the 
 inclusions) and by 
$\Op_{\Msa}$ the full subcategory of $\Op_M$ consisting of 
 subanalytic and relatively compact open subsets. 
 The site $\Msa$ is obtained by  deciding that a family  $\{U_i\}_{i\in I}$
of subobjects of $U\in\Op_{\Msa}$ is a covering of $U$ 
if there exists a finite subset $J\subset I$ such that $\bigcup_{j\in J}U_j=U$. 
One denotes by 
\eq
&&\rho_M\cl M\to\Msa
\eneq
the natural morphism of sites. Here again, we have two pairs of adjoint
functors $(\opb{\rho_M},\oim{\rho_M})$ and $(\eim{\rho_M},\opb{\rho_M})$ :

\eqn
&& \xymatrix{
{\md[\cor_M]}\ar@<-1.5ex>[rr]_-{\eim{\rho_M}}\ar@<1.5ex>[rr]^-{\oim{\rho_M}}&&{\md[\cor_\Msa]}\ar[ll]|-{\;\opb{\rho_M}\;}.
}
\eneqn
For $F\in \md[\cor_M]$, $\eim{\rho_M}F$ is the sheaf associated to the
presheaf $U\mapsto F(\overline U)$, $U\in \Op_{\Msa}$.

One proves that the restriction of $\oim{\rho_M}$ to the category $\mdrc[\cor_M]$ of $\R$-constructible sheaves 
is exact and fully faithful.
By this result, we shall consider the category $\mdrc[\cor_M]$ as  a subcategory of 
$\md[\cor_M]$ or as of $\md[\cor_{\Msa}]$. Denote by $\mdrcc[\cor_M]$ the full subcategory of $\mdrc[\cor_M]$ consisting of sheaves with compact support and set: 
$$\IIrc[\cor_M]={\rm Ind}(\mdrcc[\cor_M]).$$
One defines the functor 
\eq\label{eq:fctalphasa}
&&\alpha_{\Msa}\cl\IIrc[\cor_{M}]\To \md[\cor_{\Msa}]
\eneq
similarly as the functor $\alpha_M$.
\begin{theorem}
The functor $\alpha_{\Msa}$ in~\eqref{eq:fctalphasa}   is  
an equivalence of categories.
\end{theorem}
In other words, ind-$\R$-constructible sheaves are ``usual sheaves''
on the subanalytic site.
By this result, the embedding $\mdrcc[\cor_M]\into\mdcp[\cor_M]$ gives a 
fully faithful functor
$I_M\cl \md[\cor_{\Msa}]\to\II[\cor_M]$.
Hence, in the diagram of categories
\eq\label{dia00}
\xymatrix{
\mdrc[\cor_M]\ar[d]\ar[r]&\md[\cor_{\Msa}]\ar[d]^-{I_M}\\
\md[\cor_M]\ar@{.>}[ru]|-{\oim{\rho_M}}\ar[r]^-{\iota_M}&\II[\cor_M],
}
\eneq
all solid arrows are exact and fully faithful.
One shall be aware that the square and the upper triangle quasi-commute but
$\iota_M\not= I_M\circ \oim{\rho_M}$ in general. Moreover, 
$\oim{\rho_M}$ is not right exact in general. 

{}From now on, we shall identify sheaves on $\Msa$ with ind-sheaves on $M$.

\subsection{Ind-sheaves on bordered spaces} 
In this subsection, we  mainly follow~\cite{DK13}.

A topological space is \emph{good} if it is Hausdorff, locally compact, countable at infinity and has finite flabby dimension.

\begin{definition}\label{def:bspace}
The category of \emph{bordered spaces} is the category whose objects are pairs $(M,\bM)$ with $M\subset \bM$ an open embedding of good topological spaces.
Morphisms $f\colon (M,\bM) \to (N,\bN)$ are continuous maps $f\colon M\to N$ such that
\begin{equation}
\label{eq:Hbord}
\overline\Gamma_f \to \bM \text{ is proper}. 
\end{equation}
Here $\Gamma_f$ is the graph of $f$ and $\ol\Gamma_f$ its closure in $\bM\times\bN$.

The composition $(L,\bL) \to[e] (M,\bM) \to[f] (N,\bN)$ is given by $f\circ e\colon L\to N$ (see Lemma~\ref{lem:bordcom} below), and the identity $\id_{(M,\bM)}$ is given by $\id_{M}$. 
\end{definition}
If there is no risk of confusion, we shall often denote by $\fM$ a bordered space $(M,\bM)$.

\begin{lemma}\label{lem:bordcom}
Let $f\colon (M,\bM) \to (N,\bN)$ and $e\colon (L,\bL) \to (M,\bM)$ be morphisms of bordered spaces. Then the composition $f\circ e$ is a morphism of bordered spaces.
\end{lemma}

One shall identify a space $M$ and the bordered space $(M,M)$.
Then, by using the identifications $M=(M,M)$ and $\bM = (\bM,\bM)$, there are natural morphisms
\[
M \to (M,\bM) \to \bM.
\]
Note however that $(M,\bM) \to M$ is not necessarily a morphism of bordered spaces.

One often denotes by $j_M$ or simply $j$ the natural morphism $\fM\to\bM$.

For two bordered spaces $(M,\bM)$ and $(N,\bN)$ their product in the category of bordered spaces is the bordered space  $(M\times N,\bM\times\bN)$.

\begin{definition}\label{def:semiproper}
Let $f\colon (M,\bM) \to (N,\bN)$ be a morphism of bordered spaces. We say that $f$ is semi-proper if the map 
$\overline\Gamma_f\to\bN$ is proper.
\end{definition}
 Any isomorphism of bordered spaces is semi-proper. 
\begin{lemma}\label{lem:bordcom2}
Let $f\colon (M,\bM) \to (N,\bN)$ and $e\colon (L,\bL) \to (M,\bM)$ be morphisms of bordered spaces. If both $f$ and $e$ are semi-proper, then $f\circ e$ is semi-proper.
\end{lemma}
\begin{proof}
Set for short
$\Im(\ol{\Gamma_e}\times_{\bM}{\ol\Gamma_f})\eqdot\Im({\ol\Gamma_e}\times_{\bM}{\ol\Gamma_f}\to\ol L\times\ol N)$
and consider the  commutative diagrams:
\eqn
&&
\xymatrix{
\ol{\Gamma_e}\times_{\bM}{\ol\Gamma_f}\ar[r]^-c\ar[d]&{\ol\Gamma_f}\ar[d]\\
{\ol\Gamma_{e}}\ar@{}[ru]|-\square\ar[r]^-a&\ol M
}\hspace{8ex}\mbox{ and }\hspace{2ex}
\xymatrix{
&\ol{\Gamma_e}\times_{\bM}{\ol\Gamma_f}\ar[r]^-c\ar[d]&{\ol\Gamma_f}\ar[d]^-b\\
{\ol\Gamma_{f\circ e}}\ar[r]&\Im(\ol{\Gamma_e}\times_{\bM}{\ol\Gamma_f})\ar[r]^-d&\ol N.
}\eneqn
The diagram on the left is Cartesian. Since the map $a$ is proper,  the map $c$ is proper. Since the maps $b$ and $c$ are proper, the map $d$ is proper.
\end{proof}
Let  $\fM=(M,\bM)$ be a bordered space. Denote by 
$i\cl \bM\setminus M\to \bM$ the closed embedding. Identifying $\Derb(\cor_{\bM\setminus M})$ with its essential image in $\Derb(\cor_{\bM})$ by the fully faithful functor $\reim i \simeq \roim i$, the restriction functor 
$F\mapsto F|_M$ induces an equivalence
$\Derb(\field_{\bM}) / \Derb(\field_{\bM\setminus M})\isoto\Derb(\field_M)$.
This is no longer true for ind-sheaves. Therefore  one introduces 
\eq\label{eq:icorfM}
&&\Derb(\icor_{\fM})\eqdot \Derb(\icor_{\bM})/\Derb(\icor_{\bM\setminus M}).
\eneq
where $\Derb(\icor_{\bM\setminus M})$ is identified with its essential image in $\Derb(\icor_{\bM})$. 

The fully faithful functor $\Derb(\cor_{\bM})\into\Derb(\icor_{\bM})$ induces a fully faithful functor
\eq\label{eq:corMtoicorfM}
&&\Derb(\cor_{M})\into\Derb(\icor_{\fM}).
\eneq

We sometimes write $\Derb(\cor_{M_\infty})$ 
for the category $\Derb(\cor_{M})$ regarded as a full subcategory of 
$\Derb(\icor_{\fM})$.

Recall that if $\sht$ is a triangulated category and $\shi$ a subcategory, one denotes by ${}^\bot\shi$ and $\shi^\bot$ the left and right orthogonal to $\shi$ in $\sht$, respectively.

\begin{proposition}\label{pro:bord}
Let $\fM=(M,\bM)$ be a bordered space. One has 
\eqn
\Derb(\icor_{\bM\setminus M}) 
&=& \set{F\in\Derb(\icor_{\bM})}{\cor_M \tens F \simeq 0}\\
&=& \set{F\in\Derb(\icor_{\bM})}{ \rihom(\cor_M, F) \simeq 0 },\\
{}^\bot \Derb(\icor_{\bM\setminus M}) 
&=& \set{F\in\Derb(\icor_{\bM})}{  \cor_M \tens F\isoto F },\\
\Derb(\icor_{\bM\setminus M})^\bot &=& \set{F\in\Derb(\icor_{\bM})}{ F \isoto \rihom(\cor_M, F) }.
\eneqn
Moreover, there are  equivalences
\eqn
\Derb(\icor_{\fM}) &\isoto& \Derb(\icor_{\bM\setminus M})^\bot, \quad F \mapsto \rihom(\cor_M, F),\\
\Derb(\icor_{\fM}) &\isoto& {}^\bot \Derb(\icor_{\bM\setminus M}), \quad F\mapsto \cor_M \tens F,
\eneqn
 quasi-inverse to   the functor induced by the quotient functor.
\end{proposition}

The functors $\tens$ and $\rihom$ in $\Derb(\icor_{\bM})$ induce well defined functors (we keep the same notations)
\begin{align*}
\tens &\colon \Derb(\icor_{\fM}) \times \Derb(\icor_{\fM}) \to \Derb(\icor_{\fM}), \\
\rihom &\colon \Derb(\icor_{\fM})^\op \times \Derb(\icor_{\fM}) \to \Derb(\icor_{\fM}).
\end{align*}

Let $f\colon (M,\bM) \to (N,\bN)$ be a morphism of bordered spaces, and recall that $\Gamma_f$ denotes the graph of the associated map $f\colon M \to N$.
Since $\Gamma_f$  is  locally closed in $\bM\times \bN$, one can 
consider the sheaf $\cor_{\Gamma_f}$ on $\bM\times \bN$.

\begin{definition}\label{def:fbordered}
Let $f\colon (M,\bM) \to (N,\bN)$ be a morphism of bordered spaces.
For $F\in\Derb(\icor_{\bM})$ and $G\in\Derb(\icor_{\bN})$, one sets
\eqn
\reeim{f} F &=& \reeim{q_2}(\cor_{\Gamma_f}\tens\opb{q_1}F), \\
\roim{f} F &=& \roim{q_2}\rihom(\cor_{\Gamma_f},\epb{q_1}F), \\
\opb f G &=& \reeim{q_1}(\cor_{\Gamma_f}\tens\opb{q_2}G), \\
\epb f G &=& \roim{q_1}\rihom(\cor_{\Gamma_f},\epb{q_2}G),
\eneqn
where $q_1\colon\bM\times\bN\to\bM$ and $q_2\colon\bM\times\bN\to\bN$ are the projections.
\end{definition}

\begin{remark}
Considering a continuous map $f\colon M\to N$ as a morphism of bordered spaces with $\bM=M$ and $\bN=N$, the above functors are isomorphic to the usual operations for ind-sheaves.
\end{remark}

\begin{lemma}
\label{lem:coper}
The above definition induces well-defined functors
\begin{align*}
\reeim{f},\roim{f}  &\colon \Derb(\icor_{\fM}) \to \Derb(\icor_{\fN}), \\
\opb{f},\epb{f}  &\colon \Derb(\icor_{\fN}) \to \Derb(\icor_{\fM}).
\end{align*}
\end{lemma}
The operations for ind-sheaves on bordered spaces satisfy similar properties as for usual sheaves that we do not recall here.

Note that if $f\cl\fM\to\fN$ is semi-proper, then the diagram below quasi-commutes:
\eq\label{eq:eeimfeimf}
&&\ba{c}\xymatrix{
\Derb(\cor_M)\ar[r]^-{\reim{f}}\ar[d]&\Derb(\cor_N)\ar[d]\\
\Derb(\icor_\fM)\ar[r]^-{\reeim{f}}&\Derb(\icor_\fN).
}\ea\eneq

\subsection{Enhanced ind-sheaves}
Tamarkin's constructions of~\cite{Ta08} are extended to  ind-sheaves on bordered spaces. 
We refer to~\cite{GS12} for a detailed exposition and some complements to Tamarkin's paper. 

In this subsection, we recall results of~\cite{DK13} with a slight generalisation. In loc.\ cit.\ the authors consider 
the bordered space $M\times\R_\infty$ where $M$ is an usual space. However, we shall need to consider 
situations in which $M$ is itself a bordered space.

Consider the 2-point compactification of the real line $\ol\R \eqdot \R\sqcup\{\pm\infty\}$. Denote by $\PP\eqdot\BBP^1(\R)=\R\sqcup\{\infty\}$ the real projective line. Then $\ol\R$ has a structure of subanalytic space such that the natural map 
$\ol\R\to\PP$ is a subanalytic map.

\begin{notation}\label{not:Rinfty}
We will consider the bordered space
\[
\R_\infty \eqdot (\R,\overline\R).
\]
\end{notation}
Note that $\R_\infty$ is isomorphic to $(\R,\PP)$ as a bordered space.

Consider the morphisms of bordered spaces 
\begin{align}
\label{eq:muq1q2}
\mu,q_1,q_2 &\colon \R_\infty^2 \to \R_\infty,
\end{align}
where 
$\mu(t_1,t_2) = t_1+t_2$ and $q_1$, $q_2$ are the natural projections.

For a bordered space $\fM$, we will use the same notations for the associated morphisms
\begin{align*}
\mu, q_1,q_2 &\colon \fM\times\R_\infty^2 \to \fM\times\R_\infty.
\end{align*}
Consider also the natural morphisms
\begin{align*}
\xymatrix@C=.5em{
\fM\times\R_\infty \ar[rr]^j \ar[dr]_\pi && \fM\times\overline\R \ar[dl]^{\overline\pi} \\
&\fM\;.
}
\end{align*}

\begin{definition}\label{def:ctens1}
Let $\fM$ be a bordered space. 
The functors
\begin{align*}
\ctens &\colon \BDC(\icor_{\fM\times\R_\infty}) \times \BDC(\icor_{\fM\times\R_\infty}) \to \BDC(\icor_{\fM\times\R_\infty}), \\
\cihom &\colon \BDC(\icor_{\fM\times\R_\infty})^\op \times \BDC(\icor_{\fM\times\R_\infty}) \to \BDC(\icor_{\fM\times\R_\infty})
\end{align*}
are defined by
\begin{align*}
K_1\ctens K_2 &= \reeim {\mu} (\opb q_1 K_1 \tens \opb q_2 K_2), \\
\cihom(K_1,K_2)&= \roim {q_1} \rihom(\opb q_2 K_1, \epb{\mu}K_2).
\end{align*}
\end{definition}

One sets
\begin{align*}
\cor_{\{t\geq 0\}} &= \cor_{\set{(x,t)\in \bM\times\bR\,}{\;x\in M,\ t\in\R,\ t \geq 0}}, 
\end{align*}
and one uses similar notations for  $\cor_{\{t= 0\}}$, $\cor_{\{t> 0\}}$, $\cor_{\{t\leq 0\}}$, $\cor_{\{t< 0\}}$,  
$\cor_{\{t\not=0\}}$, etc.  
These are sheaves on $\bM\times\overline\R$ whose stalks vanish on $(\bM\times\bR)\setminus(M\times\R)$. 
We also regard them as objects of $\BDC(\icor_{\fM\times\R_\infty})$. 

The category $\BDC(\icor_{\fM\times\R_\infty})$ has a structure of commutative tensor category with $\ctens$ as tensor product and $\cor_{\{t= 0\}}$ as unit object.

One defines  the full subcategory of $\BDC(\icor_{\fM\times\R_\infty})$:
\eqn
\indc_{t^* = 0}
&=&\set{K }{\opb{\pi}\roim{\pi}K\isoto K}\\
&=&\set{K }{K\isoto \epb{\pi}\reim{\pi}K}\\
&=&\set{K}{\mbox{ there exists $L\in\Derb(\icor_\fM)$ with }\opb{\pi}L\simeq K}\\
&=& \set{K}{( \cor_{\{t\geq 0\}}\oplus \cor_{\{t\leq 0\}})  
\ctens K  \simeq 0} \\
 &=&\set{K}{ \cihom( \cor_{\{t\geq 0\}}\oplus \cor_{\{t\leq 0\}}, K)  \simeq 0}.
\eneqn

\begin{definition}\label{def:TDCforfM}
The triangulated category of {\em enhanced ind-sheaves}, denoted by  $\TDC(\icor_\fM)$, is  the quotient category
\begin{align*}
\TDC(\icor_{\fM}) &= \BDC(\icor_{{\fM}\times\R_\infty})/\indc_{t^* = 0}\\
&=\BDC(\icor_{\fM\times\R_\infty})/\set{K}{\opb{\pi}\roim{\pi}K\isoto K}.
\end{align*}
\end{definition}
Similarly, one defines the category $\TDC(\cor_{M}) $ as
\eq\label{eq:TDCforM}
&&\TDC(\cor_{M}) =\BDC(\cor_{{M}\times\R})/\set{K}{\opb{\pi}\roim{\pi}K\isoto K}.
\eneq
Then $\TDC(\cor_{M})$ is a full subcategory of $\TDC(\icor_{\fM})$.

One has the equivalences
\eq
\TDC(\icor_{\fM})& \simeq&{}^\bot\indc_{t^* = 0}\simeq(\indc_{t^* = 0})^\bot.
\eneq
Hence, we may regard $\TDC(\icor_{\fM})$ as a full subcategory of $\BDC(\icor_{\fM\times\R_\infty})$ in two different ways.

\begin{notation}\label{not:Tlr}
The functors  $\Tl,\Tr\cl\TDC(\icor_\fM)\to\BDC(\icor_{\fM\times\R_\infty})$ are given by:
\eqn
\Tl &=& (\cor_{\{t\geq0\}}\oplus\cor_{\{t\leq0\}})\ctens(\scbul), \\ 
\Tr &=&\cihom(\cor_{\{t\geq0\}}\oplus\cor_{\{t\leq0\}}, \scbul).
\eneqn
\end{notation}
The functors $\Tl$ and $\Tr$ are the left and right adjoint of the 
 quotient functor $\BDC(\icor_{\fM\times\R_\infty}) \to \TDC(\icor_\fM)$, respectively.
They are fully faithful. Note that for $K\in\TDC(\icor_\fM)$ we have
\eqn
&\roim{\pi}\Tr K\simeq0,\qquad\reeim{\pi}\Tl K\simeq0,&\\
&\roim{\pi}\Tl K\simeq \reeim{\pi}\Tr K.&
\eneqn

\subsection{Operations on enhanced ind-sheaves}
As in the preceding subsection, we recall results of~\cite{DK13} with a slight generalisation, replacing usual spaces with bordered spaces.

The bifunctors 
\begin{align*}
\ctens &\colon \TDC(\icor_\fM) \times \TDC(\icor_\fM) \to \TDC(\icor_\fM), \\
\cihom &\colon \TDC(\icor_\fM)^\op \times \TDC(\icor_\fM) \to \TDC(\icor_\fM)
\end{align*}
are those induced by the bifunctors $\ctens$ and $\cihom$ defined on $\BDC(\icor_{\fM\times\R_\infty})$.

Note that for any $K\in\TDC(\icor_\fM)$ the composition $\cor_{t\geq 0} \ctens K \to K\to \cihom(\cor_{t\geq 0} , K)$
induces an isomorphism in $\TDC(\icor_\fM)$
\[
\cor_{t\geq 0} \ctens K \isoto \cihom(\cor_{t\geq 0} , K).
\]
Such an isomorphism does not hold when replacing $\TDC(\icor_\fM)$ with $\Derb(\icor_\fM)$.

Let $f\colon \fM \to \fN$ be a morphism of bordered spaces.
Denote by $\tilde f\colon \fM\times\R_\infty \to \fN\times\R_\infty$ the associated morphism. Then the compositions of functors
\begin{align}
\label{eq:oimftilde}
\reeim{\tilde f},\; \roim{\tilde f} &\colon \BDC(\icor_{\fM\times\R_\infty}) \to \BDC(\icor_{\fN\times\R_\infty}) \to \TDC(\icor_\fN), \\
\label{eq:opbftilde}
\opb{\tilde f},\; \epb{\tilde f} &\colon \BDC(\icor_{\fN\times\R_\infty}) \to \BDC(\icor_{\fM\times\R_\infty}) \to \TDC(\icor_\fM)
\end{align}
factor through $\TDC(\icor_\fM)$ and $\TDC(\icor_\fN)$, respectively.

\begin{definition}\label{def:fT}
One denotes by
\begin{align*}
\Teeim f,\; \Toim f &\cl \TDC(\icor_\fM) \to \TDC(\icor_\fN), \\
\Topb f,\; \Tepb f &\cl \TDC(\icor_\fN) \to \TDC(\icor_\fM)
\end{align*}
the functors induced by \eqref{eq:oimftilde} and \eqref{eq:opbftilde}, respectively.
\end{definition}

The above operations satisfy analogous properties  to  the operations 
for sheaves.

One also defines similarly the external product functor
\eqn
\scbul\cetens\scbul&\cl&\TDC(\icor_\fM)\times\TDC(\icor_\fN) \to \TDC(\icor_{\fM\times \fN}),\\
&&F\cetens G=\Topb{p_1}F\ctens\Topb{p_2}G,
\eneqn
where $p_1$ and $p_2$ denote the projections from $\fM\times \fN$ to $\fM$ and $\fN$, respectively.

\begin{definition}\label{def:fihom}
One defines the hom-functor 
\begin{align}
\fihom\colon \TDC(\icor_\fM)^\op \times \TDC(\icor_\fM) &\to \BDC(\icor_\fM)\\
\fihom(K_1,K_2) &= \roim\pi\rihom(\Tl(K_1),\Tr(K_2)),\nn
\end{align}
and one sets
\eqn
&&\fhom= \alpha_M\circ\fihom\;\cl\; \TDC(\icor_\fM)^\op \times \TDC(\icor_\fM) \to \BDC(\cor_M),\\
&&\FHom(K_1,K_2) = \rsect(M;\fhom(K_1,K_2)).
\eneqn
\end{definition}
Note that 
\begin{align*}
\fihom(K_1,K_2) &\simeq \roim\pi\rihom(\Tl(K_1),\Tl(K_2)) \\
&\simeq \roim\pi\rihom(\Tr(K_1),\Tr(K_2))
\end{align*} 
and
\begin{align*}
\Hom[\TDC(\icor_\fM)](K_1,K_2)
&\simeq H^0 \bl\FHom(K_1,K_2)\br.
\end{align*}
\begin{remark}\label{rem:fiber}
For $x\in M$, let $\iota_x$ be the embedding $\rmpt\into M$ given by $\iota_x(\rmpt)=x$. 
Let $F\in\TDC(\cor_M)$. Then $F\simeq0$ if and only if $\Topb{\iota_x}F\simeq0$ for all $x\in M$. 
\end{remark}

\subsection{The functor $\cor^\Tam_M\ctens(\scbul)$}

\label{not:gg}
Consider the objects of $\BDC(\icor_{\fM\times\R_\infty})$
\[
\cor_{\{t\gg0\}} \eqdot \inddlim[a\rightarrow+\infty] \cor_{\{t\geq a\}},\qquad
\cor_{\{t<\ast\}}  \eqdot \inddlim[a\rightarrow+\infty] \cor_{\{t< a\}}.
\]
There are a distinguished triangle and isomorphisms in $\Derb(\icor_{\fM\times\R_\infty})$
\eqn
&&\cor_{M\times\R}\to \cor_{\{t\gg0\}} \to \cor_{\{t<\ast\}} \,[1]\to[+1],\\
&&\cor_{\{t\geq -a\}} \ctens \cor_{\{t\gg0\}} \isoto \cor_{\{t\gg0\}} \isoto \cor_{\{t\geq a\}} \ctens \cor_{\{t\gg0\}} ,\quad (a\in\R_{\geq 0}). 
\eneqn
Denote by $\cor^\Tam_\fM$ the object of $\TDC(\icor_\fM)$ associated with
the ind-sheaf $\cor_{\{t\gg0\}}\in\BDC(\icor_{\fM\times\R_\infty})$.
Note that 
\[
\Tl(\cor^\Tam_\fM) \simeq \cor_{\{t\gg0\}},
\quad
\Tr(\cor^\Tam_\fM) \simeq \cor_{\{t<\ast\}}[1].
\]

\begin{lemma}\label{lem:kTamtens}
For $F\in\BDC(\cor_{\fM\times\R_\infty})$ and $K\in\TDC(\icor_\fM)$, there is an isomorphism in $\TDC(\icor_\fM)$
\[
\cor_\fM^\Tam \ctens \cihom(F,K)
\isoto
\cihom(F,\cor_\fM^\Tam \ctens K).
\]
\end{lemma}

\begin{proposition}\label{pro:stableops}
Let $f\colon \fM\to \fN$ be a morphism of bordered spaces.
\bnum
\item
For $K\in\TDC(\icor_\fM)$ one has
\[
\Teeim f(\cor_\fM^\Tam \ctens K) \simeq \cor_\fN^\Tam \ctens \Teeim f K.
\]
\item
For $L\in\TDC(\icor_\fN)$ one has
\begin{align*}
\Topb f(\cor_\fN^\Tam \ctens L) &\simeq \cor_\fM^\Tam \ctens \Topb f L, \\
\Tepb f(\cor_\fN^\Tam \ctens L) &\simeq \cor_\fM^\Tam \ctens \Tepb f L.
\end{align*}
\enum
\end{proposition}
\begin{definition}\label{def:fcteepsilon}
One defines the functors 
\eq
&&e_M,\epsilon_M\cl \BDC(\icor_\fM) \to \TDC(\icor_\fM),\label{eq:eMepsilonM}\\
&&\hs{10ex}e_M(F) = \cor_\fM^\Tam\tens\opb\pi F, 
\quad\epsilon_M(F) = \cor_{\{t\geq0\}}\tens\opb\pi F.\nn
\eneq
\end{definition}
Note that
\eqn
&&e_M(F)\simeq\cor_\fM^\Tam\ctens\epsilon_M(F) .
\eneqn
\begin{proposition}\label{pro:embed}
The functors $e_M$ and $\epsilon_M$  are  fully faithful.
\end{proposition}

\subsection{Enhanced ind-sheaves over an algebra}\label{subsect:alg}
We need to generalize some definitions and results of the preceding subsections to the case where $\cor$ is replaced by a sheaf of algebras. 
For simplicity we only consider the case where $M$ is a good topological space, not a bordered space.

Let $\sha$ be
a sheaf of $\cor$-algebras on a space $X$.
We denote by $\md[\sha]$ the category of sheaves of $\sha$-modules.
Recall that an ind-sheaf of $\sha$-modules on $X$ is an ind-sheaf $\shm$
with a morphism of sheaves of $\cor$-algebras
$$\sha\to \shend(\shm).$$
We shall denote by $\II[\sha]$
the category of ind-sheaves of $\sha$-modules on $X$,
and by 
$\Derb(\II[\sha])$, or simply
$\Derb(\isha)$, the derived category of ind-sheaves of $\sha$-modules on $X$ (see~\cite{KS01}). 
 (In \cite{KS01}, it was denoted by $\mathrm{I}(\beta\sha)$ and 
$\Derb(\mathrm{I}(\beta\sha))$.) 
Then we have functors
\eqn
&&\rihom[\beta\sha](\scbul,\scbul)\cl
\Derb(\isha)^\rop\times \Derb(\isha)\To\Der(\icor_M),\\
&&\scbul\ltens[\beta\sha]\scbul\cl
\Derb(\isha^\rop)\times \Derb(\isha))\To\Der(\icor_M),
\eneqn
We write for short
\eqn
\rhom[\sha](\scbul,\scbul)&\cl&
\Derb(\md[\sha])^\rop\times \Derb(\isha)\To\Der(\icor_M)\\
\scbul\ltens[\sha]\scbul&\cl&
 \Derb(\isha^\rop)\times\Derb(\md[\sha])\To\Der(\icor_M),
\eneqn
where
\eqn
&&\rhom[\sha](\shm,K)=\rihom[\beta\sha](\beta\shm, K)\quad\text{and}\\
&&L\tens[\sha]\shm=L\ltens[\beta\sha]\beta\shm
\eneqn
for $\shm\in\Derb(\md[\sha])$, $K\in \Derb(\isha)$ and
$L\in \Derb(\md[\sha^\rop])$.

Consider a  good topological space $M$ and the bordered space $M\times\fR$. As above, we denote by 
$\pi\cl M\times\fR\to M$ and $\ol\pi\cl M\times\bR\to M$ the projections.
Assume to be given a $\cor$-flat sheaf of algebras $\sha$ on $M$.
For short, we set 
\eqn
&&\sha_{M\times\bR}=\opb{\ol\pi}\sha.
\eneqn
One defines the categories:
\eqn
\Derb(\isha_{M\times\fR})&=&\Derb(\isha_{M\times\bR})
/\set{F\in\Derb(\isha_{M\times\bR})}{F\tens\cor_{M\times\R}\simeq0},\\
\TDC(\isha)&=&\Derb(\isha_{M\times\fR})/
\set{K}{\opb{\pi}\roim{\pi}K\isoto K}.
\eneqn
We keep the notations $q_1,q_2,\mu$ as in \eqref{eq:muq1q2} and denote by
$q\cl M\times\fR\times\fR\to {M}$ the projection.

\begin{definition}\label{def:ctens2}
One defines the functors
\eq
&&
\scbul\ctens[\beta\sha]\scbul\cl
\Derb(\isha^{\;\rop}_{M\times\fR})\times\Derb(\isha_{M\times\fR})\to\Derb(\icor_{M\times\fR})\label{eq:ctensA1}\\
&&\hspace{10ex}\shn\ctens[\sha] \shm \eqdot \reeim {\mu} (\opb q_1\shn 
\lltens[\opb{q}\beta\sha] \opb q_2 \shm),\nn\\
&&
\cihom[\beta\sha](\scbul,\scbul)\cl\Derb(\isha_{M\times\fR})
\times\Derb(\isha_{M\times\fR})\to\Derb(\icor_{{M\times\fR}}) \label{eq:cihomA1} \\
&&\hspace{10ex}\cihom[\beta\sha](\shm_1,\shm_2)\eqdot \roim {q_1} \rihom[\opb{q}\beta\sha](\opb q_2 \shm_1, \epb{\mu}\shm_2). \nn
\eneq
\end{definition}
If $\sha$ is commutative, these functors take their values in 
$\Derb(\isha_{M\times\fR})$.

On easily checks that the functors~\eqref{eq:ctensA1}  and \eqref{eq:cihomA1} induce functors (we keep the same notation):
\begin{align}
\scbul\ctens[\beta\sha]\scbul&\cl\TDC(\isha^\rop)\times\TDC(\isha)\to\TDC(\icor_{M}),\label{eq:ctensA2}\\
\cihom[\beta\sha](\scbul,\scbul)&\cl\TDC(\isha)^\rop\times\TDC(\isha)\to\TDC(\icor_{M}).\label{eq:cihomA2}
\end{align}
If $\sha$ is commutative, these functors take their values in $\TDC(\isha)$.

We shall also need to consider the functors:
\begin{align}
\scbul\ltens[\sha]\scbul&\cl
\Derb(\sha^\rop)\times\Derb(\isha_{M\times\fR})\to\Derb(\icor_{M\times\fR})\label{eq:tensA1}\\
\rhom[\sha](\scbul,\scbul)&\cl\Derb(\sha)^\rop\times\Derb(\isha_{M\times\fR})\to\Derb(\icor_{{M\times\fR}})\label{eq:ihomA1}
\end{align}
which induce functors
\begin{align}
\scbul\ltens[\sha]\scbul&\cl\Derb(\sha^\rop)\times\TDC(\isha)\to\TDC(\icor_{M})\label{eq:tensB1}\\
\rhom[\sha](\scbul,\scbul)&\cl\Derb(\sha)^\rop\times\TDC(\isha)\to\TDC(\icor_{M}).\label{eq:ihomB1}
\end{align}
If $\sha$ is commutative, these functors take their values in 
$\Derb(\isha_{M\times\fR})$ or in $\TDC(\isha)$.

\section{Holomorphic solutions of $\shd$-modules}

\subsection{$\D$-modules}\label{subsectionDmod}
Reference are made to~\cite{Ka03} for the theory of $\D$-modules. The aim of this subsection is simply to fix a few notations. 

Let $(X,\sho_X)$ be a {\em complex} manifold. We introduce the following notations  (most of them are classical).
\begin{itemize}
\item
$d_X$ is the  complex dimension of $X$,
$\D_X$ the sheaf of $\C$-algebras  of holomorphic finite-order
differential operators, $\Omega_X$ the invertible sheaf of
differential forms of top degree, $\md[\D_X]$ the category of left $\D_X$-modules, $\BDC(\D_X)$ its bounded derived category.
\item
$\mop\colon \BDC(\D_X) \isoto \BDC(\D_X^\op)$ is the equivalence of categories given by
\eqn
&&\shm^\mop = \Omega_X\ltens[\O_X]\shm.
\eneqn
\item
 $\Dtens$ and $\Detens$ are the (derived) operations of tensor product and external product for $\D$-modules. Recall that $\shn\Dtens\shm=\shn\ltens[\sho_X]\shm$ in $\Derb(\sho_X)$.
\item
For $f\colon X\to Y$ a morphism of complex manifolds, 
$\Dopb f$ and $\Doim f$ are the (derived) operations of  inverse image and direct images for $\D$-modules. 
\item
The dual of $\shm\in\BDC(\D_X)$ is given by
\[
\Ddual_X\shm = \rhom[\D_X](\shm,\D_X\tens[\O_X]\Omega_X^{\otimes-1})[d_X].
\]
\item
 $\BDC_\coh(\D_X)$, $\BDC_\qgood(\D_X)$ and $\BDC_\good(\D_X)$ are the full
subcategories of $\BDC(\D_X)$ whose objects have coherent, quasi-good and good cohomologies, respectively.
Here, a $\D_X$-module $\shm$ is called \emph{quasi-good} if, for any relatively compact open subset $U\subset X$, $\shm\vert_U$
is a sum of coherent $(\O_X\vert_U)$-submodules. A $\D_X$-module $\shm$ is called \emph{good} if it is quasi-good and coherent.
\item
For $\shm\in\BDC_\coh(\D_X)$, $\chv(\shm)$ is its characteristic variety,  a closed conic involutive subset of the cotangent bundle $T^*X$. If $\chv(\shm)$ is Lagrangian, $\shm$ is called holonomic. For the notion of regular holonomic $\D_X$-module, refer e.g.\ to \cite[\S5.2]{Ka03}.
 We denote by $\BDC_{\hol}(\D_X)$ and $\BDC_{\reghol}(\D_X)$  the full 
subcategories of $\BDC(\D_X)$ whose objects have holonomic and regular
holonomic cohomologies, respectively. 
\end{itemize}

Note that $\BDC_\coh(\D_X)$, $\BDC_\qgood(\D_X)$, $\BDC_\good(\D_X)$,  $\BDC_{\hol}(\D_X)$ and $\BDC_{\reghol}(\D_X)$ are triangulated categories. 

If $Y\subset X$ is a closed hypersurface, denote by $\O_X(*Y)$ the sheaf of meromorphic functions with poles at $Y$. It is a regular holonomic $\D_X$-module. For $\shm\in\BDC(\D_X)$, set
\[
\shm(*Y) = \shm \Dtens \O_X(*Y).
\]

\subsection{Tempered functions and distributions}
In this subsection we recall some results of~\cite{KS96,KS01}.
Here,  $M$ is a real analytic manifold and $\cor=\C$.  As usual, we denote by 
$\shc_M^\infty$ (resp.\ $\shc_M^\omega$) the sheaf of  $\C$-valued functions
of class ${\mathrm C}^\infty$ (resp.\ real analytic), by $\Db_M$ 
(resp.\ $\shb_M$) 
the sheaf of Schwartz's distributions (resp.\ Sato's hyperfunctions),
and by $\shd_M$ the sheaf of  real analytic finite-order differential operators.

\begin{definition}
Let  $U$ be an open subset of $M$ 
and $f\in \shc^{\infty}_M(U)$. One says that $f$ has {\em  polynomial growth} at $p\in M$
if it satisfies the following condition.
For a local coordinate system
$(x_1,\dots,x_n)$ around $p$, there exist
a sufficiently small compact neighborhood $K$ of $p$
and a positive integer $N$
such that
\eq\label{eq:deftp}
&\sup_{x\in K\cap U}\big(\dist(x,K\setminus U)\big)^N\vert f(x)\vert
<\infty\,,&
\eneq
with the convention that if $K\cap U=\emptyset$ or if $K\subset U$, then the left-hand side of~\eqref{eq:deftp} is understood to be $0$.
It is obvious that $f$ has polynomial growth at any point of $U\cup(M\setminus\ol U)$.
We say that $f$ is {\em tempered} at $p$ if all its derivatives
have polynomial growth at $p$. We say that $f$ is tempered 
if it is tempered at any point.
\end{definition}

For an open subanalytic set $U$ in $M$,
denote by $\Cinft[M](U)$ the subspace
of $\shc^{\infty}_M(U)$ consisting of tempered $\mathrm{C}^\infty$-functions. 

Denote by $\Dbt_M(U)$ the space of tempered distributions on $U$, 
defined by the exact sequence
\eqn
&&0\to\sect_{M\setminus U}(M;\Db_M)\to\sect(M;\shd b_M)\to\Dbt_M(U)\to 0.
\eneqn                     

Using  Lojasiewicz's inequalities, one easily proves  the following results.
\begin{itemize}
\item
The presheaf  $\Cinft[M]\eqdot U\mapsto \Cinft[M](U)$ is a sheaf on  $\Msa$, hence an ind-sheaf on $M$. One calls it the {\em ind-sheaf of tempered $\mathrm{C}^\infty$-functions}.
\item
The presheaf  $\Dbt_M\eqdot U\mapsto \Dbt_M(U)$ is a sheaf on $\Msa$, hence an ind-sheaf on $M$. One calls it the {\em ind-sheaf of tempered distributions}.
\end{itemize}
Let $F\in\Derb_\Rc(\C_M)$. One has the isomorphism
\eq\label{eq:thom}
&&\alpha_M\rihom(F,\Dbt_{M})\simeq \thom(F,\Db_M),
\eneq
where the right-hand side was defined by Kashiwara  as the main tool for his proof of the Riemann-Hilbert correspondence in~\cite{Ka80,Ka84}. 
\begin{definition}[{See~\cite[Def.~5.4.1]{DK13}}]\label{def:Rbspace}
The category of \emph{real analytic bordered spaces} is defined as follows.\\
The objects are pairs $(M,\bM)$ where $\bM$ is a real analytic manifold and $M\subset \bM$ is an open subanalytic subset.\\
Morphisms $f\colon (M,\bM) \to (N,\bN)$ are real analytic maps $f\colon M\to N$ such that
\bnum
\item
$\Gamma_f$ is a subanalytic subset of $\bM\times\bN$, 
\item
$\overline\Gamma_f\to\bM$ is proper.
\ee
\end{definition}
Hence a morphism of real analytic bordered spaces is a morphism of bordered spaces.
Recall that $j_M\colon(M,\bM)\to\bM$ 
denotes the natural morphism.

\begin{definition}\label{def:CwDbtonfM}
Let $\fM=(M,\bM)$ be a real analytic bordered space. One sets
$\Dbt_{\fM} \eqdot\opb{j_M}\Dbt_{\bM}$.
\end{definition}
If $f\colon\fM\to\fN$ is an isomorphism of real analytic bordered spaces, then
$\Dbt_{\fM} \simeq \opb f\Dbt_{\fN}$ as object of $\Derb(\iC_\fM)$.

 We say that $S$ is a subanalytic subset of $\fM$
if $S$ is a subset of $M$ subanalytic in $\bM$. 

\subsection{Holomorphic functions with tempered growth}

References for this subsection are made to~\cite{KS01}. We slightly change our notations and write 
$\rhom[\D_X](\shm,\scbul)$ instead of $\rhom[\beta_X\D_X](\beta_X\shm,\scbul)$.

One defines the ind-sheaf of tempered holomorphic functions 
$\Ot[X]$ as the Dolbeault complex with coefficients in  
$\Cinft[X]$. More precisely, denoting by $X^c$ the complex conjugate manifold to $X$ and by $X_\R$ the underlying real analytic manifold, we set:
\eq\label{eq:defOt1}
\Ot[X]&=&\rhom[\shd_{X^c}](\sho_{X^c},\Cinft[X_\R]).
\eneq
One proves the isomorphism
\eq\label{eq:defOt2}
\Ot[X]&\simeq&\rhom[\shd_{X^c}](\sho_{X^c},\Dbt_{X_\R}).
\eneq
Note that the object $\Ot[X]$ is not
concentrated in degree zero in dimension $>1$. Indeed, with the subanalytic topology, only finite coverings 
are allowed. If one considers for example the open set $U\subset\C^n$, the difference of an open ball of radius $R>0$ and a closed ball of radius $0<r<R$, then the Dolbeault complex will not be exact after
 any finite covering. 
 
 Still denote by $\sho_X$ the image of this sheaf in  $\md[\iC_X]$. 
We have then the morphism in the category  $\Derb(\iC_X)$:
\eqn
&&\Ot[X]\to\sho_X.
\eneqn
  
\begin{example}
Let $Z$ be a closed complex analytic subset of the complex manifold $X$ and let $M$ be a real analytic manifold 
such that $X$ is a complexification of $M$. We have the isomorphisms
\eqn
&&\rhom[{\iC_X}](\C_Z,\Ot)\simeq\rsect_{[Z]}(\sho_X)\mbox{ (algebraic cohomology)},\\
&&\rhom[{\iC_X}](\C_Z,\sho_X)\simeq\rsect_{Z}(\sho_X),\\
&&\rhom[{\iC_X}](\RD_X'\C_M,\Ot)\simeq\Db_M,\\
&&\rhom[{\iC_X}](\RD_X'\C_M,\sho_X)\simeq\shb_M,
\eneqn
where $\RD'_X\cl \Derb(\C_X)^\rop\to\Derb(\C_X)$ is the duality functor $\rhom(\scbul,\C_X)$. 
\end{example}
Notice that with this approach, the sheaf $\Db_M$ of Schwartz's distributions is constructed similarly as the sheaf $\shb_M$ of Sato's hyperfunctions.  

The classical de Rham and solution functors are given by
\begin{align*}
\dr_X &\cl \BDC(\D_X) \to \BDC(\C_X), &\shm &\mapsto \Omega_X \ltens[\D_X] \shm, \\
\sol_X &\colon \BDC(\shd_X)^\op \to \BDC(\C_X), &\shm &\mapsto \rhom[\D_X] (\shm,\sho_X),
\end{align*}
and the tempered de Rham and solution functors are
\begin{align*}
\drt_X &\cl \BDC(\D_X) \to \BDC(\iC_X), &\shm &\mapsto \Ovt_X \ltens[\D_X] \shm, \\
\solt_X &\cl \BDC(\D_X)^\op \to \BDC({\iC_X}), &\shm& \mapsto \rhom[\D_X] (\shm,\Ot[X]).
\end{align*}
One has
\[
\sol_X \simeq \alpha_X\solt_X,
\quad
\dr_X \simeq \alpha_X\drt_X,
\]
and it follows from~\cite{Ka84} that for $\shl\in\BDC_\reghol(\D_X)$ one has
\[
\solt_X(\shl) \simeq \sol_X(\shl),
\quad
\drt_X(\shl) \simeq \dr_X(\shl).
\]
For $\shm\in\BDC_\coh(\D_X)$, one has
\[
\solt_X(\shm) \simeq \drt_X(\Ddual_X\shm)[-d_X].
\]

Let us recall some functorial properties of the tempered de Rham and solution functors.

\begin{theorem}[{\cite[Theorems 7.4.1, 7.4.6 and 7.4.12]{KS01}}]
\label{thm:ifunct}
Let $f\colon X\to Y$ be a morphism of complex manifolds. 
\bnum
\item
There is an isomorphism in $\BDC(\II[\opb f\D_Y])$
\[
\epb f \Ot[Y][d_Y] \simeq \D_{Y\leftarrow X} \ltens[\D_X] \Ot[X] [d_X].
\]
\item
For any $\shn\in\BDC(\D_Y)$, there is an isomorphism in $\BDC(\iC_X)$
\[
\drt_X(\Dopb f\shn) [d_X] \simeq 
\epb f \drt_Y(\shn) [d_Y].
\]
\item
Let $\shm\in\BDC_\good(\D_X)$, and assume that $\supp\shm$ is proper over $Y$. Then there is an isomorphism in $\BDC(\iC_Y)$
\[
\drt_Y(\Doim f\shm) \simeq 
\reeim f\drt_X(\shm) .
\]
\item
Let $\shl\in\BDC_\reghol(\D_X)$. 
Then there is an isomorphism in $\BDC(\iD_X)$
\[
\Ot[X] \ltens[\O_X] \shl \simeq 
\rihom(\sol_X(\shl),\Ot[X]).
\]
In particular, for a closed hypersurface $Y\subset X$, one has
\[
\Ot[X] \ltens[\O_X] \O_X(*Y) \simeq \rihom(\C_{X\setminus Y},\Ot[X]).
\]
\enum
\end{theorem}

\subsection{Enhanced solutions of $\D$-modules}
References for this subsection are made to~\cite{DK13}.

Let $X$ be a complex analytic manifold,  $Y\subset X$ a complex analytic hypersurface and set $U=X\setminus Y$.
For $\varphi\in\O_X(*Y)$, one sets
\begin{align*}
\D_X \e^\varphi &= \D_X/\set{P}{P\e^\varphi=0 \text{ on } U}, \\
\she^\varphi_{U|X}&=\D_X \e^\varphi(*Y).
\end{align*}
Hence $\D_X \e^\varphi$ is a $\D_X$-submodule of $\she^\varphi_{U|X}$ and 
 $\she^\varphi_{U|X}$ is a holonomic $\D_X$-module.
Moreover
\eq\label{eq:dualexpphi}
(\Ddual_X\she^\varphi_{U|X})(*Y) \simeq \she^{-\varphi}_{U|X} .
\eneq

For $c\in\R$, set for short
\[
\{\Re \varphi < c\}\seteq\set{x\in U}{\Re \varphi(x) < c} \subset X.
\]

\begin{notation}\label{not:<?}
One sets
\begin{align*}
\C_{\{\Re \varphi <\ast\}} &\eqdot \inddlim[c\rightarrow+\infty]\C_{\{\Re \varphi < c\}} \in \Derb(\iC_X), \\
E^\varphi_{U|X} &\eqdot \rihom(\C_U,\C_{\{\Re \varphi <\ast\}}) \in \BDC(\iC_X).
\end{align*}
\end{notation}

The next result (see~\cite[Prop.~6.2.2]{DK13}) generalizes~\cite[Proposition~7.3]{KS03} in which the case 
$X=\C$ and $\varphi(z)=1/z$ was treated. 

\begin{proposition}\label{pro:Solphi}
Let $Y\subset X$ be a closed complex analytic hypersurface, and set $U=X\setminus Y$.
For $\varphi\in\O_X(*Y)$, there is an isomorphism in $\BDC(\iC_X)$
\[
\drt_X(\she^{-\varphi}_{U|X}) \simeq E^\varphi_{U|X}[d_X].
\]
\end{proposition}

Recall that we have set $\PP\eqdot\BBP^1(\R)$. In the sequel, one  sets for short 
\eqn
&&\BBP\eqdot\BBP^1(\C). 
\eneqn
We denote by $\tau\in\C\subset\BBP$ the affine coordinate such that $\tau|_\R = t$, the affine coordinate of $\R$.

Consider the natural morphism of bordered spaces
\[
i\colon X\times\fR \to X\times\BBP.
\]
Recall that  $\mop\colon\BDC(\D_\BBP) \to \BDC(\D_\BBP^\op)$ is the functor given by 
$\shm^\mop = \Omega_\BBP \ltens[\O_\BBP] \shm$.

\begin{definition}\label{def:OE}
One sets:
\begin{align*}
\OEn_X 
&= \epb i ((\she_{\C|\BBP}^{-\tau})^\mop\ltens[{\D_\BBP}]\Ot[X\times\BBP])[1] 
\simeq \epb i \rhom[\D_\BBP](\she_{\C|\BBP}^\tau,\Ot[X\times\BBP])[2] \in \TDC(\iD_X), \\
\OvE_X &= \Omega_X \ltens[\O_X] \OEn_X
\simeq \epb i(\Ovt_{X\times\BBP}\ltens[\D_\BBP] \she_{\C|\BBP}^{-\tau})[1]  \in \TDC(\iD_X^\rop).
\end{align*}
One defines the {\em enhanced de Rham functor} and the 
{\em enhanced solution functor} by 
\eqn
\drE_{X}&\cl& \Derb(\D_{X})\to\TDC(\iC_X),\quad \shm\mapsto \OvE_{X}\ltens[\D_X]\shm,\\
\solE_{X}&\cl& \Derb(\D_{X})^\rop\to\TDC(\iC_X),\quad
\shm\mapsto\rhom[\D_X](\shm,\OEn_{X}).
\eneqn
One defines similarly the functors $\drE_X$ and $\solE_X$ for right modules.
\end{definition}
Note that
\[
\solE_X(\shm) \simeq \drE_X(\Ddual_X\shm)[-d_X]\quad\text{ for
$\shm\in \Derb_\coh(\D_{X})$.}
\]

\begin{theorem}\label{thm:OTgeq0}
There is an isomorphism in $\BDC(\iC_{X\times\R_\infty})$
\eqn
&&\Tr\OEn_X \simeq \epb i \bl(\she_{\C|\BBP}^{-\tau})^\mop\ltens[\D_\BBP]\Ot[X\times\BBP]\br[1],
\eneqn
and there are isomorphisms in $\TDC(\iD_X)$
\eqn
\OEn_X &\isoto& \cihom(\C_{\{t\geq 0\}}, \OEn_X) \\
&\isofrom& \cihom(\C_{\{t\geq a\}}, \OEn_X) \quad
\text{for any $a\geq 0$}\\
&\simeq& \cihom(\C_X^\enh, \OEn_X)\simeq \C_X^\Tam \ctens \OEn_X.
\eneqn
\end{theorem}

Let $\phi\in\sho_X(*Y)$ be as above. 
By \cite[Corollary 9.4.12]{DK13}, 
one has:
\eq\label{eq:laplaceEphi}
&&\solE_X(\she^{-\varphi}_{U|X})\simeq\C_X^\Tam\ctens\C_{\{t=\Re\phi\}}\simeq
\inddlim[a\to+\infty]\C_{\{t\geq\Re\phi+a\}}.
\eneq
 Here $\{t=\Re \phi\}$ denotes 
$\set{(x,t)\in M\times\R}{x\in U,\,t=\Re \phi(x)}$
and similarly for $\{t\ge\Re\phi+a\}$.

One can recover $\Ot[X]$ from $\OEn_X$. Indeed, one has:
\begin{proposition}\label{pro:fromOEtoEt}
For $F\in\Derb(\C_X)$, one has  the isomorphisms in $\Derb(\iC_X)$
\eqn
\rihom(F,\Ot[X])&\simeq&
\fihom(\C_{\{t\geq0\}}\tens\opb{\pi}F,\OEn_X)\\
&\simeq&\fihom(\C_{\{t=0\}}\tens\opb{\pi}F,\OEn_X)\\
&\simeq&\fihom(\C_X^\Tam\tens\opb{\pi}F,\OEn_X).
\eneqn
In particular, we have
\eqn
&&\Ot[X]\simeq\fihom(\C_{\{t\geq0\}},\OEn_X).
\eneqn
 \end{proposition}

For $\shm\in\Derb_\qgood(\D_X)$ recall that one sets
\eqn
&&\sol_X(\shm) = \rhom[\D_X](\shm,\O_X)\in\Derb(\C_X).
\eneqn
The next result of~\cite{DK13}  follows from Theorem~\ref{thm:ifunct}. 

\begin{theorem}\label{thm:Tfunct}
Let $f\colon X\to Y$ be a morphism of complex manifolds.
\bnum
\item\label{thm:Tfunctitem1}
There is an isomorphism in $\TDC(\II[\opb f\D_Y])$
\[
\Tepb f \OEn_Y\,[d_Y] \simeq \D_{Y\from X} \ltens[\D_X] \OEn_X \,[d_X].
\]
\item\label{thm:Tfunctitem2}
Let $\shn\in\BDC(\D_Y)$. There is an isomorphism in $\TDC(\iC_X)$, functorial in $\shn${\rm:}
\[
\drE_X(\Dopb f \shn)\,[d_X] \simeq \Tepb f \drE_Y(\shn)\,[d_Y].
\]
If moreover  $\shn\in\BDC_\hol(\D_Y)$, there is an isomorphism in $\TDC(\iC_X)$
\[
\solE_X(\Dopb f \shn) \simeq \Topb f \solE_Y(\shn).
\]
\item\label{thm:Tfunctitem3}
Let $\shm\in\BDC_\qgood(\D_X)$, and assume that $\supp\shm$ is proper over $Y$. 
There  is an  isomorphism in $\TDC(\iC_Y)$, functorial in $\shm${\rm:}
\begin{align*}
\drE_Y(\Doim f\shm) &\simeq  \Teeim f\drE_X(\shm).
\end{align*}
If moreover $\shm\in \BDC_\good(\D_X)$, then 
$$\solE_Y(\Doim f\shm)\,[d_Y]\simeq  \Teeim f\solE_X(\shm)\,[d_X].$$
\item\label{thm:Tfunctitem4}
Let $\shl\in\BDC_\reghol(\D_X)$ and $\shm\in\BDC(\D_X)$. Then
\[
\drE_X(\shl\Dtens\shm) \simeq \rihom(\opb\pi\sol_X(\shl), \drE_X(\shm)).
\]
\enum
\end{theorem}
The next result will be of constant use in the next section.

\begin{lemma}\label{le:Mtenscommut}
Let $\shm\in\Derb(\D^\rop_X)$,  $\shl\in\BDC_\hol(\D_X)$, 
$\shk\in\TDC(\mathrm{I}\D_X)$ and assume that 
$\cihom(\C_X^\Tam,\shk)\simeq \shk$. 
Then we have the natural isomorphism
\eqn
\shm\ltens[\D_X]\cihom(\solE_X(\shl),\shk)\isoto\cihom(\solE_X(\shl),\shm\ltens[\D_X]\shk).
\eneqn
\end{lemma}
\begin{proof}
Since the morphism is well-defined, we may argue locally. 
We know  by~\cite{DK13}  that there exists an object $F\in\Derb(\C_{X\times\R})$  such that
\eqn
&&\solE_X(\shl)\simeq F\ctens\C_X^\Tam.
\eneqn
It follows from the hypothesis on $\shk$ that
\eqn
&&\cihom(\solE_X(\shl),\shk)\simeq \cihom(F, \shk),\\
&&\cihom(\solE_X(\shl),\shm\ltens[\D_X]\shk)\simeq \cihom(F, \shm\ltens[\D_X]\shk).
\eneqn
 Indeed, we have
$$\cihom(\solE_X(\shl),\shm\ltens[\D_X]\shk)\simeq \cihom\bl F, 
\cihom(\C_X^\Tam,\shm\ltens[\D_X]\shk)\br.$$
By \cite[Proposition~4.7.5]{DK13}, we have
$\C_X^\Tam\ctens (\shm\ltens[\D_X]\shk)\simeq
\shm\ltens[\D_X](\C_X^\Tam\ctens\shk)\simeq\shm\ltens[\D_X]\shk$,
and $\cihom(\C_X^\Tam,\shm\ltens[\D_X]\shk)\simeq\shm\ltens[\D_X]\shk$.

\medskip
Hence, it is enough to prove the isomorphism
\eqn
&&\shm\ltens[\D_X]\cihom(F,\shk)\isoto\cihom(F,\shm\ltens[\D_X]\shk).
\eneqn
Let $\bar{q_k}\cl X\times\bR\times\bR\to X\times\bR$ ($k=1,2$)
be the projection,
$\sigma\cl X\times\R\times\R\to X\times\R$ the map $(t_1,t_2)\mapsto t_2-t_1$
and $i\cl X\times\R\times\R\to X\times\ol\R\times\ol \R$  the inclusion.
Then we have
\eqn
\shm\ltens[\D_X]\cihom(F,\shk)&\simeq &
\shm\ltens[\D_X]\reeim{\bar{q_1}}\rihom(\reim{i}\sigma^{-1}F,{\bar{q_2}}^!\shk)\\
&\underset{(1)}{\simeq}&
\reeim{\bar{q_1}}\Bigl(\shm\ltens[\D_X]\rihom(\reim{i}\sigma^{-1}F,{\bar{q_2}}^!\shk)\Bigr)\\
&\underset{(2)}{\simeq}&\reeim{\bar{q_1}}\rihom(\reim{i}\sigma^{-1}F,\shm\ltens[\D_X]{\bar{q_2}}^!\shk)\\
&\underset{(3)}{\simeq}&\reeim{\bar{q_1}}\rihom\bl\reim{i}\sigma^{-1}F,
{\bar{q_2}}^!(\shm\ltens[\D_X]\shk)\br\\
&\simeq&
\cihom(F,\shm\ltens[\D_X]\shk).
\eneqn
Here the isomorphisms
(1), (2) and (3) follow from
Theorem 5.5.4, Theorem 5.6.1 (ii) and Theorem 5.6.3
in \cite{KS01}, respectively.
\end{proof}

\section{Integral transform for De Rham}

\subsection{An enhanced Riemann-Hilbert correspondence}

Theorem~7.4.12 of~\cite{KS01}  (that is, Theorem~\ref{thm:ifunct} (iv)
in this paper)  is a reformulation of 
a result of Bj\"ork~\cite{Bj93}.
The aim of this subsection is to extend  this result 
to the case where $\shl$ is holonomic but not necessarily regular.
Note that we do not any more use Bj\"ork's result in the proof of Theorem~\ref{th:bjork} below. 
As we shall see, this last theorem generalizes the reconstruction theorem (Riemann-Hilbert) of~\cite{DK13} but of course, its proof deeply uses the tools of loc.\ cit.

\begin{lemma}\label{le:Topb}
Let $\iota\cl X\to Y$ be a closed embedding of complex manifolds. There is a natural isomorphism
\eqn
&&\Topb{\iota}(\shd_{X\to Y}\lltens[\shd_Y]\OEn_Y)\simeq\OEn_X.
\eneqn
\end{lemma}
\begin{proof}
Applying Theorem~\ref{thm:Tfunct}~(\ref{thm:Tfunctitem3}) with $\shm=\shd_X$ we get
\eqn
&&\Toim{\iota}\OEn_X\simeq\shd_{X\to Y}\lltens[\shd_Y]\OEn_Y,
\eneqn
and the result follows. 
\end{proof}

\begin{lemma}\label{le:canmor1}
There is a canonical morphism in $\TDC(\iD_X)$
\eqn
&&\OEn_X\ctens[\beta\sho_X]\OEn_X\to \OEn_X.
\eneqn
\end{lemma}
\begin{proof}
Consider the diagram in which $q_k$ and $\tw p_k$ ($k=1,2$) denote the projections. 
\eqn
&&\xymatrix{
X&\ar[l]_-qX\times\R_\infty\times\R_\infty\ar[d]_-{\tw\delta}\ar@<1.5ex>[rr]^{q_k}\ar@<-1.5ex>[rr]^\mu&&X\times\R_\infty\ar[d]_-\delta\\
&X\times\R_\infty\times X\times\R_\infty\ar[rr]^-{\tw\mu}\ar[d]_-p\ar[rrd]_-{\tw p_k}&&X\times X\times\R_\infty\\
&X\times X&&X\times\R_\infty.
}\eneqn
Denote by $q\cl X\times\fR\times\fR\to X$ the projection. 
One has
\begin{align*}
\opb{q_1}\Tr\OEn_X\ltens[\opb{q}\beta\sho_X]\opb{q_2}\Tr\OEn_X&\\
&\hspace{-4.5ex}\simeq (\opb{q_1}\Tr\OEn_X\tens\opb{q_2}\Tr\OEn_X)\ltens[\opb{q}(\sho_X\tens\sho_X)]\opb{q}\sho_X\\
&\hspace{-4.5ex}\simeq \opb{\tw\delta}\bl\opb{\tw p_1}\Tr\OEn_X\tens\opb{\tw p_2}\Tr\OEn_X\br\ltens[\opb{q}(\sho_X\tens\sho_X)]\opb{q}\sho_X.
\end{align*}
On the other hand, by~\cite[Prop.8.2.4]{DK13}, there is a canonical morphism
\eqn
\OEn_X\cetens\OEn_X&\simeq&\reeim{\tw\mu}(\opb{\tw p_1}\Tr\OEn_X\tens\opb{\tw p_2}\Tr\OEn_X)\\
&\to& \OEn_{X\times X}.
\eneqn
By adjunction, we get the morphism
\eqn
\opb{\tw p_1}\Tr\OEn_X\tens\opb{\tw p_2}\Tr\OEn_X&\to&\epb{\tw\mu}\Tr \OEn_{X\times X}.
\eneqn
Therefore, we have the morphisms
\eqn
\opb{\tw q_1}\Tr\OEn_X\ltens[\opb{q}\beta\sho_X]\opb{\tw q_2}\Tr\OEn_X&&\\
&&\hspace{-7ex}\simeq\opb{\tw\delta}(\opb{\tw p_1}\Tr \OEn_{X}\tens\opb{\tw p_2}\Tr \OEn_{X})\ltens[\opb{q}(\sho_X\tens\sho_X)]\opb{q}\sho_{X}\\
&&\hspace{-7ex}\to\opb{\tw\delta}\epb{\tw\mu}\Tr \OEn_{X\times X}\ltens[\opb{q}(\sho_X\tens\sho_X)]\opb{q}\sho_X\\
&&\hspace{-7ex}\to\opb{\tw\delta}\epb{\tw\mu}\bl\Tr\OEn_{X\times X}\ltens[\opb{p}\sho_{X\times X}]\opb{p}\sho_X\br\\
&&\hspace{-7ex}\simeq\epb{\mu}\opb{\delta}\bl\Tr\OEn_{X\times X}\ltens[\opb{p}\sho_{X\times X}]\opb{p}\sho_X\br\\
&&\hspace{-7ex}\simeq\epb{\mu}\Tr\OEn_X,
\eneqn
where $\sho_X$ is identified  with  $\sho_{\Delta_X}$ by the diagonal embedding 
$\delta\cl X\into X\times X$, and the last isomorphism follows from 
Lemma~\ref{le:Topb}.

The result then follows by adjunction. 
\end{proof}

Let $M$ be a good topological space, 
$\cor$ a commutative unital ring  and  $\sha$ a flat $\cor$-algebra on $M$ as above.

\begin{lemma}\label{le:inddproo}
Let $M\in\Derb(\sha)$, $F\in\Derb(\isha^\rop)$ 
and  $G\in\Derb(\isha)$. Then there is a canonical morphism
in $\Derb(\icor_M)${\rm:}
\eqn
&&(F\lltens[\sha]M)\ltens[\cor]\rhom[\sha](M,G)\to F\ltens[\beta\sha]G.
\eneqn 
\end{lemma}
\begin{proof}
We have 
\eqn
&&(F\lltens[\sha]M)\ltens[\cor]
\rhom[\sha](M,G)\\
&&\hs{20ex}=(F\lltens[\beta\sha]\beta M)\ltens[\cor]\rihom[\beta\sha](
\beta M,G)\\
&&\hs{20ex}\simeq F\lltens[\beta\sha]\bl\beta M\ltens[\cor]\rihom[\beta\sha](\beta M,G)\br\\
&&\hs{20ex}\to F\ltens[\beta\sha]G.
\eneqn
\end{proof}

\begin{lemma}\label{lem:MtoIhomSolM}
There exists a canonical morphism functorial with respect to $\shm\in\Derb(\D_X)$:
\eq\label{eq:bjorkmor}
&&\shm\Dtens\OEn_X\to\cihom(\solE_X(\shm),\OEn_X)\mbox{ in }\TDC(\iD_X).
\eneq
\end{lemma}
\begin{proof}
It is enough to construct the morphism
\eqn
(\shm\Dtens\OEn_X)\ctens\solE_X(\shm)&\to&\OEn_X.
\eneqn
With the notations as in Lemma~\ref{le:canmor1} we have 
\begin{align*}
(\shm\Dtens\OEn_X)\ctens\solE(\shm)
&\simeq\reeim{\mu}\bl\opb{q_1}(\beta\shm\ltens[\beta\sho_X]\Tr\OEn_X)\tens\opb{q_2}\solE_X(\shm)\br\\
&\hspace{-18ex}\simeq\reeim{\mu}\bl\opb{q_1}(\Tr\OEn_X\ltens[\beta\sho_X]
\beta\D_X\ltens[\beta\D_X]\beta\shm)\tens\opb{q_2}
\rihom[\beta\D_X](\beta\shm,\Tr\OEn_X)\br\\
&\hspace{-18ex}\to\reeim{\mu}\bl\opb{q_1}(\Tr\OEn_X\ltens[\beta\sho_X]\beta\D_X)\ltens[\opb{q}\beta\D_X]\opb{q}\beta\shm\\
&\hs{15ex}\tens\rihom[q^{-1}\beta\D_X](q^{-1}\beta\shm,
q_2^{_1}\Tr\OEn_X)\br\\
&\hspace{-18ex}\to\reeim{\mu}\bl(\opb{q_1}\Tr\OEn_X
\ltens[\opb{q_1}\beta\sho_X]\opb{q_1}\beta\D_X)\ltens[\opb{q}\beta\D_X]\opb{q_2}\Tr\OEn_X\br\\
&\hspace{-18ex}\simeq\reeim{\mu}\bl\opb{q_1}\Tr\sho_X\ltens[\opb{q}\beta\sho_X]\opb{q_2}\Tr\OEn_X\br\\
&\hspace{-18ex}\simeq\OEn_X\ctens[\beta\sho_X]\OEn_X\to\OEn_X.
\end{align*}
Here, we have used Lemma~\ref{le:inddproo}.
\end{proof}

\begin{theorem}[{\rm Extended Riemann-Hilbert theorem}]\label{th:bjork}
There exists a canonical isomorphism functorial with respect to $\shm\in\Derb_\hol(\D_X)${\rm:}
\eq\label{eq:bjorkmorf}
&&\shm\Dtens\OEn_X\isoto\cihom(\solE_X(\shm),\OEn_X)\mbox{ in }\TDC(\iD_X).
\eneq
\end{theorem}
\begin{proof}
We shall apply \cite[Lemma~7.3.7]{DK13}. 
Recall that this lemma summarizes deep results of Mochizuki~\cite{Mo09,Mo11} in the algebraic case completed by those of Kedlaya~\cite{Ke10,Ke11} in the analytic case (see also~\cite{Sa00}). 

All conditions of this lemma are easily verified except conditions~(e) and~(f) that we shall check now. 

\vspace{1ex}\noindent
(e) Let $f\cl X\to Y$ be a projective morphism and let $\shm$ be a good holonomic $\D$-module
such that \eqref{eq:bjorkmor} is an isomorphism. We shall prove the isomorphism:
\eq\label{eq:cond(e)}
\Doim{f}\shm \Dtens\OEn_Y\isoto \cihom(\solE_Y(\Doim{f}\shm),\OEn_Y).
\eneq
By Theorem~\ref{thm:Tfunct}~(\ref{thm:Tfunctitem1}) we have 
\eqn
&&  \D_{Y\from X}\ltens[\D_X] \OEn_X\simeq\Tepb{f}\OEn_Y\,[d_Y-d_X].
\eneqn
Therefore, using Lemma~\ref{le:Mtenscommut}, we get:
\begin{align*}
\D_{Y\from X}\ltens[{\D_X}] &\cihom(\solE_X(\shm),\OEn_X)\\
& \simeq \cihom(\solE_X(\shm),\Tepb{f}\OEn_Y\,[d_Y-d_X]).
\end{align*}
 Hence, using Theorem~\ref{thm:Tfunct}~(\ref{thm:Tfunctitem3}) :
\eqn
&& \Toim{f}\bl \D_{Y\from X}\ltens[\D_X] \cihom(\solE_X(\shm),\OEn_X)\br\\
&&\hs{25ex}\simeq\cihom(\Teeim{f}\solE_X(\shm),\OEn_Y\,[d_Y-d_X])\\
&&\hspace{25ex}\simeq\cihom(\solE_Y(\Doim{f}\shm),\OEn_Y).
\eneqn
Hence, the right hand side of \eqref{eq:cond(e)} is calculated as
\eq
&&\ba{l}\cihom(\solE_Y(\Doim{f}\shm),\OEn_Y)\\[1ex]
\hs{20ex}\simeq\Toim{f}\bl\D_{Y\from X}\ltens[\D_X] \cihom(\solE_X(\shm),\OEn_X) \br. 
\ea\label{eq:cond(e)R}\eneq
On the other hand we have:
\eqn
\OvE_Y \ltens[\sho_Y] \Doim{f}\shm&\simeq&\OvE_Y\ltens[\D_Y]\D_Y \Dtens\Doim{f}\shm\\
&\simeq&\drE_Y(\D_Y\Dtens\Doim{f}\shm)\\
&\simeq&\drE_Y(\Doim{f}(\D_{X\to Y}\Dtens\shm)),
\eneqn
where the last isomorphism follows from  the projection formula of~\cite[Th.~4.2.8]{Ka03}. 
Hence 
\eqn
\OvE_Y\ltens[\sho_Y]\Doim{f}\shm&\simeq&\drE_Y\bl\Doim{f}(\D_{X\to Y}\Dtens\shm)\br\\
&\simeq&\Toim{f}\bl\drE_X(\D_{X\to Y}\Dtens\shm)\br,
\eneqn
where the last isomorphism follows from Theorem~\ref{thm:Tfunct}~(\ref{thm:Tfunctitem3}).  Since
\eqn
\drE_X(\D_{X\to Y}\Dtens\shm)&\simeq&\OvE_X\ltens[\D_X](\D_{X\to Y}\Dtens\shm)\\
&\simeq&(\OvE_X\ltens[\sho_X]\shm)\ltens[\D_X]\D_{X\to Y},
\eneqn
we obtain
\eqn
&&\OvE_Y\ltens[\sho_Y]\Doim{f}\shm\simeq\Toim{f}\bl
(\OvE_X\ltens[\sho_X]\shm)\ltens[\D_X]\D_{X\to Y}\br,
\eneqn
which is equivalent to
\eq
&&\Doim{f}\shm\Dtens \OEn_Y\simeq\Toim{f}\bl
\D_{Y\from X}\ltens[\D_X](\shm\Dtens\OEn_X)\br.
\label{eq:cond(e)L}
\eneq
By comparing~\eqref{eq:cond(e)R} and \eqref{eq:cond(e)L} and assuming that \eqref{eq:bjorkmor} is an isomorphism,
we get  isomorphism~\eqref{eq:cond(e)}.

\vspace{1ex}\noindent
(f)--(1)\quad Let $Y\subset X$ be a normal crossing divisor, $U=X\setminus Y$ and 
$\phi\in\sect(X;\sho_X(*Y))$ a meromorphic function on $X$  with poles on $Y$.
We shall first prove that \eqref{eq:bjorkmor} is an isomorphism when $\shm=\she^\varphi_{U|X}$.

Recall that $\BBP=\BBP^1(\C)$ and denote by $\BBP_1$ and $\BBP_2$ two copies of $\BBP$. Let $\tau_k$ be the inhomogeneous coordinate of $\BBP_k$ ($k=1,2$).
 Let $q_{kX}\cl X\times\R_\infty\times \R_\infty\to  X\times\R_\infty$ be the
projection ($k=1,2$).
Consider the maps
\eqn
&&\xymatrix@C=10ex{
X\times\R_\infty\times\R_\infty\ar[r]^-j\ar[d]_-{q_{1X}}
&X\times\BBP_1\times\R_\infty\ar[d]^-{\pi_{X\times\BBP_1}}\ar[r]^-{i_{X\times\BBP}}\ar[rd]^q\ar@{}[dl]|-{\square}& X\times\BBP_1\times\BBP_2\\
 X\times\R_\infty\ar[r]^-{i_X}  &  
X\times\BBP_1&X\times \R_\infty
}\eneqn
Set $K\eqdot\solE_X(\shm)$. Then by~\cite[Prop.~9.6.5]{DK13}, we have
\eqn
\roim{\pi_{X\times\BBP_1}}\rihom(\opb{q}\Tl(K),\Tr\OEn_{X\times\BBP_1})\simeq\shm\Dtens\Ot[X\times\BBP_1].
\eneqn
On the other hand, we have
\eqn
&&\epb{i_X}\roim{\pi_{X\times\BBP_1}}\rihom(\opb{q}\Tl(K),\Tr\OEn_{X\times\BBP})\\
&&\hspace{10ex}\simeq
\roim{q_{1X}}\epb{j}\rihom(\opb{q}\Tl(K),\Tr\OEn_{X\times\BBP})\\
&&\hspace{10ex}\simeq\roim{q_{1X}}\rihom(\opb{j}\opb{q}\Tl(K),\epb{j}\Tr\OEn_{X\times\BBP_1}).
\eneqn
Since
\eqn
\Tr\OEn_{X\times\BBP_1}&\simeq&\epb{i_{X\times\BBP}}\bl(\she^{-\tau_2}_{\C\vert\BBP_2})^\mop\ltens[\D_{\BBP_2}]
\Ot[X\times\BBP_1\times\BBP_2]\br[1],\eneqn
we obtain
\eqn
&&\epb{i_X}(\shm\Dtens\Ot[X\times\BBP_1])\\
&&\hs{3ex}\simeq\roim{q_{1X}}\rihom\Bigl(\opb{j}\opb{q}\Tl(K),
\epb{j}\epb{i_{X\times\BBP}}\bl(\she^{-\tau_2}_{\C\vert\BBP_2})^\mop
\ltens[\D_{\BBP_2}]
\Ot[X\times\BBP_1\times\BBP_2]\br[1]\Bigr).
\eneqn
By applying the functor $(\she^{-\tau_1}_{\C|\BBP_1})^\mop\ltens[\D_{\BBP_1}]\scbul$, we get the isomorphism
\eqn
&&\hs{-3ex}\shm\Dtens\Tr\OEn_{X}\simeq\epb{i_X}\Bigl((\she^{-\tau_1}_{\C|\BBP_1})^\mop\ltens[\D_{\BBP_1}](\shm\Dtens\Ot[X\times\BBP_1])\Bigr)[1]\\
&&\hs{-1ex}\simeq
\roim{q_{1X}}\rihom\Bigl(\opb{q_{2X}}\Tl(K),
\epb{j}\epb{i_{X\times\BBP_1}}((\she^{-\tau_1-\tau_2}_{\C^2|\BBP_1\times\BBP_2})^\mop\ltens[\D_{\BBP_1\times\BBP_2}]\Ot[X\times\BBP_1\times\BBP_2])[2]\Bigr).
\eneqn
Let $\tilde\mu\cl X\times(\C,\BBP_1)\times (\C,\BBP_2)\to X\times(\C,\BBP)$ be the morphism
given by $(\tau_1,\tau_2)\mapsto\tau_1+\tau_2$,
and let $\mu\cl X\times\R_\infty\times \R_\infty\to X\times\R_\infty$
be its restriction.
Since
\eqn
(\she^{-\tau_1-\tau_2}_{\C^2|\BBP_1\times\BBP_2})^\mop\ltens[\D_{\BBP_1\times\BBP_2}]\Ot[X\times\BBP_1\times\BBP_2][2]
&\simeq&
(\Dopb{\tilde\mu}\she^{-\tau}_{\C|\BBP})^\mop\ltens[\D_{\BBP_1\times\BBP_2}]\Ot[X\times\BBP_1\times\BBP_2][2]\\
&\simeq&\epb{\tilde\mu}\Bigl((\she^{-\tau}_{\C|\BBP})^\mop\ltens[\D_\BBP]
\Ot[X\times\BBP]\Bigr)[1],
\eneqn
we get
\eqn
&&\epb{j}\epb{i_{X\times\BBP}}\Bigl((\she^{-\tau_1-\tau_2}_{\C^2|\BBP_1\times\BBP_2})^\mop\ltens[\D_{\BBP_1\times\BBP_2}]\Ot[X\times\BBP_1\times\BBP_2])\Bigr)
[2]\\
&&\hs{15ex}\simeq\epb{\mu}\epb{i_X}\Bigl((\she^{-\tau}_{\C|\BBP})^\mop\ltens[\D_\BBP]
\Ot[X\times\BBP]\Bigr)[1]\simeq
\epb{\mu}\Tr\OEn_X.
\eneqn
Hence we obtain
\eqn
\shm\Dtens\OEn_X\simeq
\shm\Dtens\Tr\OEn_X&\simeq&
\roim{q_{1X}}\rihom(\opb{q_{2X}}\Tl(K),\epb{\mu}\Tr\OEn_X)\\
&\simeq&\cihom( K,\OEn_X).
\eneqn

\vspace{1ex}\noindent
(f)--(2)\quad Let $Y\subset X$ be a normal crossing divisor, $U=X\setminus Y$ as in (f)--(1). We shall prove that 
\eqref{eq:bjorkmor} is an isomorphism when $\shm$ has a normal form along $Y$. 
We keep the notations of~\cite[Lemma~9.6.6]{DK13}. 
Recall from \cite[\S~7.1, 7.2]{DK13} that $\twX\eqdot\twX_Y$ 
is the real blow-up of $X$ along $Y$ 
and $\olom\cl\twX\to X$ is the projection.

 Similarly to Lemma~\ref{lem:MtoIhomSolM}, we have a morphism
\eq\label{lem:MtoIhomSolMb}
&&\shl\ltens[\sha_{\twX}]\OEn_\twX\to\cihom(\solE_\twX(\shl),\OEn_\twX)
\quad\text{for any }\shl\in\Derb(\D^\sha_\twX).
\eneq
Let us  first  show that~\eqref{lem:MtoIhomSolMb} is an isomorphism when $\shl=\shm^\sha$. 
Note that by~\cite[(9.6.6)]{DK13}, 
\eq
&&\opb{\pi}\opb{\olom}\C_U\tens\solE_\twX(\shm^\sha)\simeq \Topb{\olom}
\solE_{ X}(\shm).\label{eq:Solb}
\eneq
Since $\shm^\sha$ is isomorphic to a direct sum of modules of type $(\she^\phi_{U|X})^\sha$, we may assume that $\shm^\sha=\shn^\sha$ with 
$\shn=\she^\phi_{U|X}$. In this case we have proved the isomorphism
\eq\label{lem:MtoIhomSolMc}
&&\shn\ltens[\sho_X]\OEn_X\isoto\cihom(\solE_{X}(\shn),\OEn_{X})
\eneq
 in (f)--(1).     By~\cite[Th.~9.2.2]{DK13}, 
\eqn
&&\OEn_\twX\simeq\Tepb{\olom}\rihom(\opb{\pi}\C_U,\OEn_X).
\eneqn
Applying the functor $\Tepb{\olom}\rihom(\opb{\pi}\C_U,\scbul)$ to \eqref{lem:MtoIhomSolMc},
 we obtain
\eqn
&&\shm^\sha\ltens[\sha_{\twX}]\OEn_\twX\simeq \shn\Dtens\OEn_\twX   
\isoto\cihom(\solE_{\twX}(\shm^\sha),\OEn_{\twX}).
\eneqn
Therefore  \eqref{lem:MtoIhomSolMb}  is an isomorphism  for $\shl=\shm^\sha$.

Hence, we obtain
\eqn
&&\shm^\sha\ltens[\sho_{\tw X}]\Tepb{\olom}\rihom
(\opb{\pi}\C_U,\OEn_X)
\simeq \shm^\sha\ltens[\sha_{\twX}]\OEn_\twX\\
&&\hspace{10ex}\simeq
\cihom(\solE_{\twX}(\shm^\sha),\Tepb{\olom}\rihom(\opb{\pi}\C_U,\OEn_X))\\
&&\hspace{10ex}\simeq
\cihom(\opb{\pi_\twX}\opb{\olom}\C_U\tens\solE_{\twX}(\shm^\sha),\Tepb{\olom}\OEn_X)\\
&&\hspace{10ex}\simeq
\cihom(\Topb{\olom}\solE_{X}(\shm),\Tepb{\olom}\OEn_X)\\
&&\hspace{10ex}\simeq
\Tepb{\olom}\cihom(\solE_{X}(\shm),\OEn_X).
\eneqn
Here, we have used  \eqref{eq:Solb}.  
Applying the functor $\Toim{\olom}$ we obtain
\begin{align*}
\shm\Dtens\rihom&(\opb{\pi}\C_U,\OEn_X)\\
&\simeq\rihom(\opb{\pi}\roim{\olom}\C_\twX,\cihom(\solE_X(\shm),\OEn_X))\\
&\simeq\cihom(\opb{\pi}\roim{\olom}\C_\twX\ltens \solE_X(\shm),\OEn_X).
\end{align*}
Since $\opb{\pi}\C_U\tens\solE_X(\shm)\simeq\solE_X(\shm)$, we also have 
$\opb{\pi}\roim{\olom}\C_\twX\tens\solE_X(\shm)\simeq\solE_X(\shm)$.

Since 
$\rihom(\opb{\pi}\C_U,\OEn_X)\simeq\sho_X(*Y)\tens[\sho_X]\OEn_X$ and $\shm\tens[\sho_X]\sho_X(*Y)\simeq\shm$,
we obtain isomorphism~\eqref{eq:bjorkmorf}.

Thus condition (f) is checked, and this completes the proof 
of  isomorphism~\eqref{eq:bjorkmorf} for an arbitrary $\shm\in\Derb_\hol(\D_X)$.
\end{proof}

\begin{corollary}[{ cf.~\cite[Th.~9.6.1]{DK13}}]
There are isomorphisms, functorial with respect to $\shm\in\Derb_\hol(\D_X)$%
{\rm:}
\begin{align*}
\shm\Dtens\Ot[X]&\isoto\fihom(\solE_X(\shm),\OEn_X)
\quad\text{in $\Derb(\mathrm{I}\shd_X)$,}\\
\shm&\isoto\fhom(\solE_X(\shm),\OEn_X)\quad\text{in $\Derb(\shd_X)$.}
\end{align*}
\end{corollary}
\begin{proof}
We shall apply the functor $\fihom(\C_{\{t=0\}},\scbul)$ 
to isomorphism~\eqref{eq:bjorkmorf}. 
The first isomorphism  follows by Proposition~\ref{pro:fromOEtoEt}. The second isomorphism is deduced from the first one by applying the functor $\alpha_X$.
\end{proof}

\subsection{Functoriality of  the De Rham functor}
\begin{proposition}\label{pro:opforenhD}
The enhanced de Rham functor has the following properties.
\bnum
\item
Let $f\cl X\to Y$ be a morphism of complex manifolds and let $\shn\in\Derb_\qgood(\D_{Y})$. There is a natural isomorphism
\eqn
&&\drE_{X}(\Dopb{f}\shn)\,[d_X]\simeq\Tepb{f}\drE_{Y}(\shn)\,[d_Y].
\eneqn
\item
Let $f\cl X\to Y$ be a morphism of complex manifolds and let $\shm\in\Derb_\qgood(\D_{X})$. 
Assume that $f$ is proper on $\supp(\shm)$. Then there is a natural isomorphism
\eqn
&&\drE_{Y}(\Doim{f}\shm)\simeq\Toim{f}\drE_{X}(\shm).
\eneqn
\item
Let $X$ be a complex manifold and let $\shl\in\Derb_\hol(\D_X)$ and 
 $\shm\in\Derb(\D_{X})$.  There is a natural isomorphism
\[
\drE_X(\shl\Dtens\shm) \simeq \cihom(\solE_X(\shl), \drE_X(\shm)).
\]
\enum
\end{proposition}
\begin{proof}
(i) and (ii) follow from Theorem~\ref{thm:Tfunct}. Let us prove~(iii).
By Theorem~\ref{th:bjork}, we have an isomorphism  in $\TDC(\iD_X)$:
\eq\label{eq:bjorkmorfbis}
&&\shl\Dtens\OEn_X\isoto\cihom(\solE_X(\shl),\OEn_X).
\eneq
Let us apply $\shm^\mop\ltens[\shd_X]\scbul$ to both sides of \eqref{eq:bjorkmorfbis}. We have
\eqn
\shm^\mop\ltens[\shd_X](\shl\Dtens\OEn_X)&\simeq&(\shm\Dtens\shl)^\mop\ltens[\shd_X]\OEn_X\\
&\simeq&\drE_X(\shm\Dtens\shl),
\eneqn
and, using Lemma~\ref{le:Mtenscommut},
\eqn
\shm^\mop\ltens[\shd_X]\cihom(\solE_X(\shl),\OEn_X)&\simeq&\cihom(\solE_X(\shl),\shm^\mop\ltens[\shd_X]\OEn_X)\\
&\simeq&\cihom(\solE_X(\shl),\drE_X(\shm)).
\eneqn
\end{proof}
Consider morphisms of complex manifolds
\eq\label{diag:inttrans}
&&\ba{c}\xymatrix@C=6ex{
&S\ar[ld]_-f\ar[rd]^-g&\\
{X}&&{Y}.
}\ea\eneq

\begin{notation}\label{not:PhiPsiE}
(i) For $\shm\in\Derb_\qgood(\D_{X})$ and $\shl\in\Derb_\qgood(\D_S)$ one sets
\eq\label{eq:Dcconv}
&&\shm\Dconv\shl\eqdot\Doim{g}(\Dopb{f}\shm\Dtens\shl).
\eneq
(ii) For $L\in \TDC(\iC_S)$,  $F\in \TDC(\iC_{X})$ and $G\in\TDC(\iC_Y)$ one sets
\eq\label{eq:Ccconv}
&&\ba{c} L\econv G\eqdot\Teeim{f}(L\ctens\Topb{g}G),\quad
 F\econv L\eqdot\Teeim{g}(\Topb{f}F\ctens L),\\
\Phi_L^\Tam(G)=  L\econv G,\quad \Psi_L^\Tam(F)=\Toim{g}\cihom(L,\Tepb{f}F).
\ea\eneq
\end{notation}
Note that we have a pair of adjoint functors 
\eq\label{eq:phipsiadj}
&&\xymatrix{
\Phi_L^\Tam\cl \TDC(\iC_Y)\ar@<-0.5ex>[r]&\TDC(\iC_X)\ar@<-0.5ex>[l]\cl \Psi_L^\Tam
}\eneq

\begin{theorem}\label{th:7412}
Let $\shm\in\Derb_\qgood(\D_X)$, 
$\shl\in\Derb_\ghol(\D_{S})$ and let  $L\eqdot\solE_{S}(\shl)$. 
 Assume that $\opb{f}\supp(\shm)\cap\supp(\shl)$ is proper over $Y$.
Then there is a natural isomorphism in $\TDC(\iC_Y)${\rm:}
\eqn
&&\Psi_L^\Tam\bl\drE_X(\shm)\br\,[d_X-d_S]\simeq\drE_Y(\shm\Dconv\shl).
\eneqn
\end{theorem}
\begin{proof}
Applying Proposition~\ref{pro:opforenhD}, we get:
\eqn
\drE_Y(\shm\Dconv\shl)&=& \drE_Y(\Doim{g}(\Dopb{f}\shm\Dtens\shl))\\
&\simeq&\Toim{g} \drE_S(\Dopb{f}\shm\Dtens\shl)\\
&\simeq&\Toim{g}  \cihom(\solE_S(\shl), \drE_S(\Dopb{f}\shm))\\
&\simeq&\Toim{g}  \cihom(L, \Tepb{f}\drE_X(\shm))\,[d_X-d_S]\\
&=&\Psi_L^\Tam(\drE_X(\shm))\,[d_X-d_S].
\eneqn
\end{proof}

\begin{corollary}\label{cor:7412}
In the situation of {\rm Theorem~\ref{th:7412}}, let $G\in \TDC(\iC_Y)$. Then there is a natural isomorphism in $\Derb(\C)$
\eqn
&&\FHom(L\econv G,\OvE_{X}\ltens[\D_{X}]\shm)\,[d_X-d_S]\\
&&\hspace{30ex}\simeq\FHom(G,\OvE_{Y}\ltens[\D_{Y}](\shm\Dconv \shl)).
\eneqn
\end{corollary}
\begin{proof}
Applying Theorem~\ref{th:7412} and the adjunction~\eqref{eq:phipsiadj}, we get
\eqn
\FHom(\Phi_L^\Tam(G), \drE_X(\shm))&\simeq&\FHom(G, \Psi_L^\Tam(\drE_X(\shm)))\\
&\simeq&\FHom(G, \drE_Y(\shm\Dconv\shl)) \,[d_S-d_X].
\eneqn
\end{proof}
Note that Corollary~\ref{cor:7412} is a generalisation of
\cite[Th.7.4.12]{KS01} to not necessarily regular  holonomic $\D$-modules.

\subsection{Generalization to complex bordered spaces}
It is possible to generalize all the preceding results when replacing complex manifolds with complex bordered spaces, as defined below. 
\begin{definition}\label{def:Cbspace}
The category of \emph{complex bordered spaces} is defined as follows.\\
The objects are pairs $(X,\bX)$ where $\bX$ is a complex manifold and $X\subset \bX$ is an open  subset such that $\bX\setminus X$ is a complex analytic subset of $\bX$.\\
Morphisms $f\colon (X,\bX) \to (Y,\bY)$ are complex analytic maps $f\colon X\to Y$ such that
\bnum
\item
$\Gamma_f$ is a complex analytic subset of $\bX\times\bY$,
\item
$\overline\Gamma_f\to\bX$ is proper.
\ee
\end{definition}
Hence a morphism of complex  bordered spaces is a morphism of real analytic bordered spaces
(cf. Definition~\ref{def:Rbspace}).

Let $\fX=(X,\bX)$ be a complex bordered space. One sets 
\begin{align*}
&\Derb(\D_{\fX})\eqdot\Derb(\D_\bX)/\set{\shm}{\supp(\shm)\subset \bX\setminus X}
\simeq\Derb(\D_X),\\
&\Derb_\qgood(\D_{\fX})\eqdot\Derb_\qgood(\D_\bX)/\set{\shm}{\supp(\shm)\subset \bX\setminus X},\\
&\Derb_\good(\D_{\fX})\eqdot\Derb_\good(\D_\bX)/\set{\shm}{\supp(\shm)\subset \bX\setminus X},\\
&\Derb_\hol(\D_{\fX})\eqdot   \Derb_\hol(\D_\bX)/\set{\shm}{\supp(\shm)\subset \bX\setminus X},\\
&\Derb_\ghol(\D_{\fX})\eqdot\Derb_\ghol(\D_{\bX})/\set{\shm}{\supp(\shm)\subset \bX\setminus X}.
\end{align*}
Let $f\cl \fX=(X,\bX)\to\fY=(Y,\bY)$ be a morphism of complex bordered spaces. One sets
\eqn\label{eq:Bf}
&&\shb_f\eqdot \rsect_{[\Gamma_f]}(\sho_{\bX\times \bY})\,[d_Y].
\eneqn
Denote by $p_1$ and $p_2$ the first and second projection defined on $\bX\times\bY$. 
By representing an object of  $\Derb(\D_{\fX})$ by an object of $\Derb(\D_{\bX})$ and similarly 
with $\Derb(\D_{\fY})$, we define the functors:
\begin{align*}
\Dopb{f}\cl& \Derb(\D_{\fY})\to \Derb(\D_{\fX}),\quad
\shn\mapsto \Doim{{p_1}}(\Dopb{p_2}\shn\Dtens\shb_f),\\
\Doim{f}\cl& \Derb(\D_{\fX})\to \Derb(\D_{\fY}),\quad
\shm\mapsto \Doim{{p_2}}(\Dopb{p_1}\shm\Dtens\shb_f).
\end{align*}
Also set:
\eq\label{eq:Df}
&&\shd_{\fX\to\fY}\eqdot \shb_f\tens[\opb{p_2}\sho_{\bY}]\opb{p_2}\Omega_{\bY},\quad
\shd_{\fY\from\fX}\eqdot \shb_f\tens[\opb{p_1}\sho_{\bX}]\opb{p_1}\Omega_{\bX}.
\eneq
\begin{lemma}
Let $\shn\in \Derb(\D_{\fY})$ and let $\shm\in \Derb(\D_{\fX})$.Then
\eqn
\Dopb{f}\shn\simeq \roim{p_1}(\shd_{\fX\to\fY}\ltens[\opb{p_2}\D_{\bY}]\opb{p_2}\shn),\\
\Doim{f}\shm\simeq \roim{p_2}(\shd_{\fY\from\fX}\ltens[\opb{p_1}\D_{\bX}]\opb{p_1}\shm).
\eneqn
\end{lemma}
\begin{proof}
After replacing the notations $X$ and $Y$ by $\bX$ and $\bY$, 
this follows from the formula, valid for $\shm\in\Derb(\D_{X})$ and $\shl\in\Derb(\D_{X\times Y})$:
\eqn
&&\D_{Y\from X\times Y}\lltens[\D_{X\times Y}](\opb{p_1}\shm\lltens[\opb{p_1}\sho_X]\shl)\simeq 
\opb{p_1}\shm^\mop\lltens[\opb{p_1}\D_X]\shl.
\eneqn
\end{proof}

\begin{lemma}
\banum
\item
The functor $\Dopb{f}$ above induces  well-defined functors 
\begin{align*}
\Dopb{f}\cl&\Derb_\qgood(\D_{\fY})\to \Derb_\qgood(\D_{\fX}),\\
\Dopb{f}\cl&\Derb_\hol(\D_{\fY})\to \Derb_\hol(\D_{\fX}).
\end{align*}
\item Assume that the morphism $f$ is semi-proper. 
Then the functor $\Doim{f}$ above induces well-defined functors 
\begin{align*}
\Doim{f}&\cl \Derb_\qgood(\D_{\fX})\to \Derb_\qgood(\D_{\fY}),\\
\Doim{f}&\cl\Derb_\ghol(\D_{\fX})\ \to \Derb_\ghol(\D_{\fY}).
\end{align*}
\eanum
\end{lemma}

 The proof is obvious.

\begin{definition}\label{def:OEbis}
Let $\fX=(X,\bX)$ be a complex bordered space and denote by $j\cl\fX\to\bX$ the natural morphism. We set
\eq\label{eq:OEbis}
&&\OEn_\fX\eqdot\Topb{j}\OEn_\bX.
\eneq
We define similarly $\OvE_\fX$ and we define the functors 
\begin{align*}
\drE_\fX&\cl\Derb_\qgood(\D_\fX)\to\TDC(\iC_\fX),\\
\solE_\fX&\cl\Derb_\qgood(\D_\fX)^\rop\to\TDC(\iC_\fX),
\end{align*}
as in Definition~\ref{def:OE}. 
\end{definition}

\begin{proposition}\label{pro:opforenhDbis}
\bnum
\item
Let $f\cl \fX\to \fY$ be a morphism of complex bordered spaces and let $\shn\in\Derb_\qgood(\D_{\fY})$. There is a natural isomorphism
\eqn
&&\drE_{\fX}(\Dopb{f}\shn)\,[d_X]\simeq\Tepb{f}\drE_{\fY}(\shn)\,[d_Y].
\eneqn
\item
Let $f\cl \fX\to \fY$ be a morphism of complex bordered spaces and let $\shm\in\Derb_\qgood(\D_{\fX})$. 
Assume that $f$ is semi-proper. Then there is a natural isomorphism
\eqn
&&\drE_{\fY}(\Doim{f}\shm)\simeq\Toim{f}\drE_{\fX}(\shm).
\eneqn
\item
Let $\fX$ be a complex bordered space and let $\shl\in\Derb_\hol(\D_\fX)$ and $\shm\in\Derb_\qgood(\D_{\fX})$. There is a natural isomorphism
\[
\drE_\fX(\shl\Dtens\shm) \simeq \cihom(\solE_\fX(\shl), \drE_\fX(\shm)).
\]
\enum
\end{proposition}
\begin{proof}
There exist a complex manifold $Z$ and a proper morphism $h\cl Z\to\ol\Gamma_f$ such that $\opb{h}X\to X$ is an isomorphism. Hence, replacing $(X,\bX)$ with $(X,Z)$ we may assume from the beginning that $f\cl X\to Y$ extends to 
a morphism of complex manifolds $\baf\cl\bX\to\bY$. 

\vspace{0.3ex}\noindent
(i) Choose a representative $\shn'\in\Derb_\qgood(\D_{\bY})$ of $\shn$ and apply 
Proposition~\ref{pro:opforenhD}~(i). 

\vspace{0.3ex}\noindent
(ii) Choose a representative $\shm'\in\Derb_\qgood(\D_{\bX})$ of $\shm$.  
Then $\Doim{f}\shm$ is represented by 
$\Doim{\baf}(\shm'\Dtens\rsect_{[X]}(\sho_{\bX}))$.
We have
\eqn
\drE_\bX(\shm'\Dtens\rsect_{[X]}(\sho_{\bX}))&\simeq& 
\Toim{j_X}\Topb{j_X}\drE_\bX(\shm')\\
&\simeq&\Toim{j_X}\drE_\fX(\shm).
\eneqn
Applying Proposition~\ref{pro:opforenhD}~(ii), we get
\eqn
\drE_\fY(\Doim{f}\shm)&\simeq&\Toim{\baf}\drE_\bX
(\shm'\Dtens\rsect_{[X]}(\sho_{\bX}))\\
&\simeq&\Toim{\baf}\Toim{j_X}\Topb{j_X}\drE_\bX(\shm')\\
&\simeq&\Toim{f}\drE_\fX(\shm).
\eneqn

\vspace{0.3ex}\noindent
(iii) Choose a representative $\shm'$ of $\shm$ as in (ii), choose a representative $\shl'$ of $\shl$ and apply 
Proposition~\ref{pro:opforenhD}~(iii) to $\shm'$ and $\shl'$.
\end{proof}

Consider morphisms of bordered spaces.
\eq\label{diag:inttransinfty}
&&\xymatrix{
&\fS\ar[ld]_-f\ar[rd]^-g&\\
{\fX}&&{\fY}\,.
}\eneq

\begin{notation}\label{not:PhiPsiEbis}
(i) For $\shm\in\Derb_\qgood(\D_{\fX})$ and $\shl\in\Derb_\qgood(\D_\fS)$ one sets
\eq\label{eq:Dcconvbis}
&&\shm\Dconv\shl\eqdot\Doim{g}(\Dopb{f}\shm\Dtens\shl).
\eneq
(ii) For $L\in \TDC(\iC_\fS)$,  $F\in \TDC(\iC_{\fX})$ and $G\in\TDC(\iC_\fY)$ one sets
\eq\label{eq:Ccconvbis}
&&\ba{c} L\econv G\eqdot\Teeim{f}(L\ctens\Topb{g}G),\\
\Phi_L^\Tam(G)=  L\econv G,\quad \Psi_L^\Tam(F)=\Toim{g}\cihom(L,\Tepb{f}F).
\ea\eneq
\end{notation}
Here again, we get a pair of adjoint functors
\eq\label{eq:phipsiadjbis}
&&\xymatrix{
\Phi_L^\Tam\cl \TDC(\iC_{\fY})\ar@<-0.5ex>[r]&\TDC(\iC_{\fX})\ar@<-0.5ex>[l]\cl \Psi_L^\Tam
}\eneq

\begin{theorem}\label{th:7412bis}
Assume that $g$ is semi-proper. Let $\shm\in\Derb_\qgood(\D_\fX)$, 
$\shl\in\Derb_\ghol(\D_{\fS})$ and let  $L\eqdot\solE_{\fS}(\shl)$. 
Then there is a natural isomorphism in $\TDC(\iC_\fY)$
\eqn
&&\Psi_L^\Tam(\drE_\fX(\shm))\,[d_X-d_S]\simeq\drE_\fY(\shm\Dconv\shl).
\eneqn
\end{theorem}
\begin{proof}
The proof goes as for Theorem~\ref{th:7412}, using Proposition~\ref{pro:opforenhDbis} instead of 
Proposition~\ref{pro:opforenhD}.
\end{proof}
Note that the preceding results generalize to complex bordered spaces 
Theorem~\ref{th:bjork}, Proposition~\ref{pro:opforenhD} and Theorem~\ref{th:7412}.

\section{Enhanced Fourier-Sato transform}
\subsection{Enhanced Fourier-Sato transform}
Let $\VV$ be a real finite-dimensional vector space, $\VVd$ its dual. 
Recall that the Fourier-Sato transform is an equivalence of categories 
between conic sheaves on $\VV$ and conic sheaves on $\VV^*$.
 References are made to~\cite{KS90}. 
In~\cite{Ta08}, D.~Tamarkin has extended the Fourier-Sato transform 
to no more conic (usual) sheaves, by adding an extra variable. 
Here we generalise this last transform to enhanced ind-sheaves on $\fV$.

We set $n=\dim\VV$ and we denote by $\ori_\VV$ the orientation $\cor$-module of $\VV$, {\em i.e.,} 
$\ori_\VV=H^n_c(\VV;\cor_\VV)$.
We have a canonical isomorphism $\ori_\VV\simeq\ori_{\VVd}$. 
We denote by $\Delta_\VV$ the diagonal of $\VV\times\VV$. 

We consider the bordered space $\fV=(\VV,\bV)$ where $\bV$ is the projective compactification of 
$\VV$, that is
\eqn
&&\bV=\bl(\VV\oplus\R)\setminus\{0\}\br/\R^\times.
\eneqn
We shall work in the categories of enhanced ind-sheaves. If $\fM$ is a bordered space and 
$F\in\Derb(\icor_\fM)$, recall that in Definition~\ref{def:fcteepsilon}  we set: 
\eq\label{def:F+}
\epsilon_M(F)\eqdot\opb{\pi}F\tens\cor_{\{t\geq0\}}\in\TDC(\cor_\fM).
\eneq
Also recall that  $\pi$ is the projection $\fM\times\fR\to\fM$ and $t$ is the coordinate on $\R$.

Recall (see~\eqref{eq:TDCforM}) that for a bordered space $\fM$, $\TDC(\cor_M)$ is a full subcategory of $\TDC(\icor_\fM)$.

We introduced the kernels
\begin{align}\label{eq:eFT1}
&&\ba{c}
L_\VV\eqdot\cor_{\{t=\langle x,y\rangle\}}\in\TDC(\cor_{\VV\times\VVd})\subset\TDC(\icor_{\fV\times\fVd}),\\
L^a_\VV\eqdot\cor_{\{t=-\langle x,y\rangle\}}\in\TDC(\cor_{\VV\times\VVd})\subset\TDC(\icor_{\fV\times\fVd}).
\ea\end{align}
Here, $x$ and $y$ denote points of $\VV$ and $\VVd$, respectively.

\begin{lemma}\label{le:compkernels}
One has isomorphisms in $\TDC(\icor_{\fV\times\fV})$
\eq\label{eq:eFT3}
&&\ba{c}
L_\VV\econv L^a_\VVd\isoto\cor_{\Delta_\VV\times\{t=0\}}\tens\ori_\VV\,[-n],\\
L^a_\VVd\econv L_\VV\isoto\cor_{\Delta_{\VVd}\times\{t=0\}}\tens\ori_\VV\,[-n].
\ea\eneq
\end{lemma}
\begin{proof}
Of course, it is enough to prove the first isomorphism. 
Denote by $(x,y,x')$ a point of $\VV\times\VVd\times\VV$ and denote by $p$ the projection
$\fV\times\fVd\times\fV\times\fR\to\fV\times\fV\times\fR$. 
We have in $\Derb(\icor_{\fV\times\fV\times\fR})$:
\eqn
L_\VV\econv L^a_\VVd&\simeq&\reeim{p}(\cor_{t=\langle x,y\rangle}\ctens\cor_{t=-\langle x',y\rangle})\\
&\simeq&\reeim{p}(\cor_{t=\langle x-x',y\rangle})\\
&\simeq&\reim{p}(\cor_{t=\langle x-x',y\rangle}),
\eneqn
where the last isomorphism follows from the fact that  $p$ is semi-proper
(see Diagram~\eqref{eq:eeimfeimf}).

The first morphism in~\eqref{eq:eFT3} is deduced from the morphism $\cor_{t=\langle x-x',y\rangle}\to\cor_{\{x=x'\}\times\{t=0\}}$.
To check it is an isomorphism it is thus enough to calculate the restriction  of these sheaves at each fiber of 
$\VV\times\VV\times\R\to\VV\times\VV$ (see Remark~\ref{rem:fiber}). 

The restriction of the left-hand side of~\eqref{eq:eFT3}. to $(x,x')\times\R$ is 
\eq
&&
\begin{cases}
\cor_\R\tens \ori_\VV\,[1-n]&\mbox{ if }x\not=x',\\
\cor_{\{t=0\}}\tens\ori_\VV\,[-n]&\mbox{ if }x=x'.\\
\end{cases}
\eneq
Since  the image of $\cor_{\R}$ in 
$\TDC(\icor_{\rmpt})$ is $0$, we get the result.
\end{proof}
Now we introduce the enhanced Fourier-Sato functors 
\begin{align}\label{eq:fourier}
&&\ba{c}
\EF[\VV]\cl \TDC(\icor_{\fV})\to\TDC(\icor_{\fVd}),\quad \EF[\VV](F)=F\econv L_\VV,\\
\EFa[\VV]\cl \TDC(\icor_{\fV})\to\TDC(\icor_{\fVd}),\quad \EFa[\VV](F)=F\econv L^a_\VV.
\ea
\end{align}
Applying Lemma~\ref{le:compkernels}, we obtain:

\begin{theorem}[{See~\cite{Ta08}}]\label{th:fourier}
The functors $\EF[\VV]$ and $\EFa[\VVd]\tens\ori_\VV\,[n]$ are equivalences of categories, inverse to each other. In other words, one has the isomorphisms, functorial with respect to $F\in \TDC(\icor_{\fV})$
and $G\in \TDC(\icor_{\fVd})${\rm:}
\eqn
&&\EFa[\VVd]\circ\EF[\VV](F)\simeq F\tens\ori_\VV\,[-n],\\
&&\EF[\VV]\circ\EFa[\VVd](G)\simeq G\tens\ori_\VV\,[-n].
\eneqn
\end{theorem}
\begin{corollary}\label{cor:fourier}
There is an isomorphism functorial in $F_1,F_2\in\TDC(\icor_\fV)${\rm:}
\eq\label{eq:isoFourier}
&&\FHom(F_1,F_2)\simeq\FHom(\EF[\VV](F_1),\EF[\VV](F_2)).
\eneq
\end{corollary}

We shall give an alternative construction of $\EF[\VV]$. Denote by $p_1$ and $p_2$ the projections from 
$\VV\times\VVd$ to $\VV$ and $\VVd$, respectively and  recall Notation~\ref{not:PhiPsiEbis}
in which $\fS=\fV\times\fVd$, $\fX=\fV$, $\fY=\fVd$, $f=p_1$ and $g=p_2$. 

\begin{corollary}\label{cor:fourieradjoint}
The two functors $\EF[\VV](\scbul)$ and  $\Psi^\Tam_{L^a_\VV}(\scbul)\tens\ori_\VV\,[-n]$ are isomorphic.
\end{corollary}
\begin{proof}
The functor   $\EFa[\VVd](\scbul)\tens\ori_\VV\,[n]$ admits an inverse, namely  the functor $\EF[\VV](\scbul)$,
and also admits  $\Psi^\Tam_{L^a_\VV}(\scbul)\tens\ori_\VV\,[-n]$ as a right adjoint. Therefore, these two last functors are isomorphic.
\end{proof}
For a bordered space $\fM$, denote by $a_\fM\cl \fM\to\rmpt$ the unique morphism from $\fM$ to the bordered space
$\rmpt$. 
\begin{corollary}\label{cor:fourierdual}
 We have the isomorphism, functorial with respect to  $F\in \TDC(\icor_{\fV})$ and $G\in\TDC(\icor_\rmpt)$:
 \eqn
 &&\EF[\VV](\cihom(F,\Tepb{a_\fV}(G)))\simeq\cihom(\EFa[\VV](F),\Topb{a_\fVd}(G))).
 \eneqn
 \end{corollary}
 \begin{proof}
 One has the sequence of  isomorphisms 
 \eqn
 &&\cihom(\EFa[\VV](F),\Topb{a_\fVd}(G)))\tens\ori_\VV\,[n]
\simeq\cihom(\EFa[\VV](F),\Tepb{a_\fVd}G)\\
 &&\hs{22ex}\simeq\cihom(\Teeim{p_2}(L^a_\VV\ctens\Topb{p_1}F),\Tepb{a_\fVd}G)\\
 &&\hs{22ex}
 \simeq\Toim{p_2}\cihom(L^a_\VV\ctens\Topb{p_1}F,\Tepb{p_2}\Tepb{a_\fVd}G)\\
&&\hs{22ex}\simeq\Toim{p_2}\cihom(L^a_\VV,\cihom(\Topb{p_1}F,\Tepb{p_1}\Tepb{a_\fV}G))\\
 &&\hs{22ex}\simeq\Toim{p_2}\cihom(L^a_\VV,\Tepb{p_1}
\cihom(F,\Tepb{a_\fV}G))\\
 & &\hs{22ex}\simeq\EF[\VV](\cihom(F,\Tepb{a_\fV}G))\tens\ori_\VV\,[n].
  \eneqn
Here the last isomorphism follows from the preceding corollary.
 \end{proof}

\subsection{Operations}
Let $f\cl\VV\to\VV'$ be a linear map of finite-dimensional vector spaces. We denote by $n$ and $n'$ the
dimensions of $\VV$ and $\VV'$, respectively. We denote by $\trf$ the transpose of $f$.
\begin{proposition}
Let $F\in\TDC(\icor_\fV)$ and let $G\in\TDC(\icor_{\VV_\infty'})$. Then{\rm:}
\begin{align}
&\EF[\VV'](\Teeim{f}F)\simeq\Topb{\trf}\EF[\VV](F),\label{eq:fouriereim}\\
&\EF[\VV](\Topb{f}G\tens\ori_{\VV'}\,[-n'])\simeq\Teeim{\trf}\EF[\VV'](G)\tens\ori_\VV\,[-n].\label{eq:fourieropb}
\end{align}
\end{proposition}
\begin{proof}
Consider the commutative diagram
\eqn
&&\xymatrix{
\VV'\times\VVdp\ar[d]_-{p'_1}\ar@/^2pc/[rrr]^-{p'_{2}}
                 &\VV\times\VVdp\ar[l]^-g\ar[d]|-{p_1\circ h}\ar[rr]^-{p'_2\circ g}\ar[rd]^-h&&\VVdp\ar[d]^-{\trf}\\
 \VVp\ar@{}[ru]|-\square&\VV\ar[l]^-f&  \VV\times\VVd\ar[l]^-{p_1}\ar@{}[u]|-\square\ar[r]_-{p_2}&\VVd\,.             
}\eneqn
\vspace{0.3ex}\noindent
(i) Let us prove~\eqref{eq:fouriereim}.
Using $\Topb{g}L_{\VV'}\simeq\Topb{h}L_\VV$, we have 
\eqn
\Teeim{p_2'}(\Topb{p_1'}\Teeim{f}F\tens L_{\VV'})
&\simeq&\Teeim{p_2'}(\Teeim{g}\Topb{(p_1\circ h)}F\tens L_{\VV'})\\
&\simeq&\Teeim{p_2'}\Teeim{g}(\Topb{h}\Topb{p_1}F\tens\Topb{g}L_{\VV'})\\
&\simeq&\Teeim{p_2'}\Teeim{g}(\Topb{h}\Topb{p_1}F\tens\Topb{h}L_{\VV})\\
&\simeq&\Teeim{(p_2'\circ g)}\Topb{h}(\Topb{p_1}F\tens L_{\VV})\\
&\simeq &\Topb{\trf}\Teeim{p_2}(\Topb{p_1}F\tens L_{\VV}).
\eneqn

\vspace{0.3ex}\noindent
(ii) Let us prove~\eqref{eq:fourieropb}. Applying~\eqref{eq:fouriereim}, we obtain
\eqn
\EF[\VVd]\Teeim{\trf}(\EF[\VV']G)
&\simeq&\Topb{f}\EFa[\VVdp]\EF[\VV']G\simeq \Topb{f}G\tens\ori_{\VV'}\,[-n'].
\eneqn
Hence
\eqn
\EF[\VV]\EFa[\VVd]\Teeim{\trf}(\EF[\VV']G)&\simeq&\EF[\VV]\Topb{f}G\tens\ori_{\VV'}\,[-n'],
\eneqn
and the result follows since 
\eqn
\EF[\VV]\EFa[\VVd]K\simeq K\tens\ori_\VV\,[-n].
\eneqn
\end{proof}

\subsection{Compatibility of Fourier-Sato  transforms}
We shall compare the enhanced Fourier-Sato transform with the classical one, for which we refer to~\cite[\S\,3.7]{KS90}.

Recall that one denotes by $\Derb_{\R^+}(\cor_\VV)$ the full subcategory of $\Derb(\cor_\VV)$ consisting of conic sheaves. We shall denote here by $\FS[\VV](F)$ the Fourier-Sato transform of 
$F\in\Derb_{\R^+}(\cor_\VV)$, which was denoted by $F^\wedge$ in loc.\ cit. The functor  
$\FS[\VV]\cl\Derb_{\R^+}(\cor_\VV)\to\Derb_{\R^+}(\cor_\VVd)$ is an equivalence of categories. 

Recall that one identifies the sheaf $\cor_{\{t\geq0\}}$ with its image in $\BDC(\icor_{\VV\times\R_\infty})$
and that the functor
\eqn
&&\epsilon_\VV\cl\BDC(\cor_\VV) \into \TDC(\icor_\fV),\quad \epsilon_\VV(F)= \cor_{\{t\geq0\}}\tens\opb\pi F
\eneqn
is a fully faithful embedding (see Proposition~\ref{pro:embed}).

Consider the diagram of categories and functors
\eq\label{diag:fourierefourier}
&&\ba{c}\xymatrix{
\TDC(\icor_\fV)\ar[rr]^-{\EF[\VV]}&&\TDC(\icor_\fVd)\\
\Derb_{\R^+}(\cor_\VV)\ar[rr]^-{\FS[\VV]}\ar[u]_-{\epsilon_\VV}&&\Derb_{\R^+}(\cor_\VVd)\ar[u]_-{\epsilon_\VVd}.
}\ea
\eneq
\begin{theorem}\label{th:fourierefourier}
Diagram~\eqref{diag:fourierefourier} is quasi-commutative.
\end{theorem}

\begin{proof}
Consider the morphism of diagrams (in which we denote by $\pi$ any of the projections 
$X\times\fR\to X$, with $X=\VV,\VVd$, etc.):
\eqn
&&\xymatrix@R=1.0ex@C=3.8ex{
&\VV\times\VVd\times\fR\ar[lddd]_-{p_1}\ar[rddd]^-{p_2}
&&&\VV\times\VVd\ar[lddd]_-{q_1}\ar[rddd]^-{q_2}&\\
&&\ar[r]^-\pi&&& \\\\
\VV\times\fR&&\VVd\times\fR&\VV&&\VVd.
}
\eneqn
(i)  We shall first construct the morphism, functorial in $F\in\Derb_{\R^+}(\cor_\VV)$:
\eq\label{eq:EFtoF}
&&\EF[\VV]\circ\epsilon_\VV(F)\to\epsilon_\VVd\circ\FS[\VV](F).
\eneq
Consider the sequence of morphisms in $\Derb(\cor_{\VVd\times\R})$:
\begin{align*}
\reim{p_2}(\cor_{\{t=\langle x,y\rangle\}}\ctens\opb{p_1}(\cor_{\{t\geq0\}}\ltens\opb{\pi}F))
&\simeq
\reim{p_2}(\cor_{\{t\geq\langle x,y\rangle\}}\tens\opb{p_1}\opb{\pi}F))\\
&\to\reim{p_2}(\cor_{\{t\geq0\geq\langle x,y\rangle\}}\tens\opb{p_1}\opb{\pi}F))\\
&\simeq\reim{p_2}(\cor_{\{t\geq0\}}\tens\cor_{\{0\geq\langle x,y\rangle\}}\tens\opb{\pi}\opb{q_1}F)\\
&\simeq\cor_{\{t\geq0\}}\tens\opb{\pi}\reim{q_2}(\cor_{\{0\geq\langle x,y\rangle\}}\tens\opb{q_1}F).
\end{align*}
The image of these morphisms in $\TDC(\icor_\VVd)$ 
gives the morphism \eqref{eq:EFtoF}. To prove that it is an isomorphism,
we again argue in  $\Derb(\cor_{\VVd\times\R})$.

By the exact sequence
\eqn
&&0\to \cor_{\{t\geq\langle x,y\rangle>0\}}\oplus\cor_{\{0>t\geq\langle x,y\rangle\}}\to \cor_{\{t\geq\langle x,y\rangle\}}\to \cor_{\{t\geq0\geq\langle x,y\rangle\}}
\to0,
\eneqn
it is enough to show that for any $(y,t)\in\VVd\times\R$, 
\eqn
&&\rsect_c(\VV;\cor_{\{0>t\geq\langle x,y\rangle\}}\tens F)\simeq0,\\
&&\rsect_c(\VV;\cor_{\{t\geq\langle x,y\rangle>0\}}\tens F)\simeq0.
\eneqn
Denote by $h\cl\VV\to\R$ the map $h(x)=\langle x,y\rangle$ and set $G=\reim{h}F\in\Derb(\cor_\R)$.  Then
\eq\label{eq:rsectcIG}
&&\ba{c}
\rsect_c(\VV;\cor_{\{0>t\geq\langle x,y\rangle\}}\tens F)\simeq\rsect_c(\{\lambda\in\R;0>t\ge\lambda\};G),\\[1ex]
\rsect_c(\VV;\cor_{\{t\geq\langle x,y\rangle>0\}}\tens F)\simeq\rsect_c(\{\lambda\in\R;0<\lambda\le t\};G).
\ea\eneq
Since $F$ is $\R^+$-conic, so is $G$ and the vanishing of the right-hand sides of~\eqref{eq:rsectcIG}
follows.
\end{proof}

\subsection{Legendre transform}
In this section, we shall make a link between the enhanced Fourier-Sato transform and the classical Legendre transform of convex functions. 

As above, $\VV$ is a real vector space of dimension $n$. 

\begin{definition}
Let $f\cl\VV\to\R\sqcup{\{+\infty\}}$ be a function.
\banum
\item
We say that $f$ is a  closed proper convex  function on $\VV$ 
if its epigraph 
$\set{(x,t)\in\VV\times\R}{t\geq f(x)}$ is closed, convex and non-empty.
\item
We denote by $\Leg[\VV]$ the space of  closed proper convex functions on $\VV$.
\item
For $f\in\Leg[\VV]$, we denote by $\dom(f)$ the set $\opb{f}(\R)$ and call it the domain of $f$. We denote by $\rmH(f)$ the affine space generated by $\dom(f)$  and by 
$\domo(f)$ the interior of $\dom(f)$ in $\rmH(f)$.
\item 
We define $f^*\cl\VVd\to\R\sqcup{\{+\infty\}}$ by 
$f^*(y)=\sup_{x\in\dom(f)}\bl\langle x,y\rangle-f(x)\br$
and call $f^*$  the Legendre transform, or 
the  Legendre-Fenchel conjugate or the convex conjugate of $f$. 
\eanum
\end{definition}
Note that
\begin{itemize}
\item
the set $\dom(f)$ is convex and non-empty,
\item
the function $f^*$ is also a closed proper convex function, that is, belongs to $\Leg[\VVd]$,
\item
$f^{**}=f$. Hence, ${}^*$ gives an isomorphism $\Leg[\VV]\isoto\Leg[\VVd]$. 
\end{itemize}
Now we introduce the set:
\begin{align*}
\rmE(f)&=\{v\in\VV\,;\;\mbox{there exists }a\in\R\mbox{ such that }f(x+v)=f(x)+a\\
&\hspace{47ex}\mbox{ for any }x\in \VV\}\\
&=\{v\in\VV\;;\mbox{ there exists }a\in\R\mbox{ such that }f(x+tv)=f(x)+ta\\
&\hspace{47ex}\mbox{ for any }x\in \VV,t\in\R\}.
\end{align*}
Denote by $\rmH^\bot$ the orthogonal vector space to an affine space $\rmH\subset\VVd$, that is,
$\rmH^\bot=\set{v\in\VVd}{\text{$v\vert_\rmH$ is constant}}$. Then
\eqn
&&\rmE(f)=\rmH(f^*)^\bot.
\eneqn
In particular, $\dim \rmE(f)=\codim\rmH(f^*)$.
In the sequel, we set:
\eq
&& \rmd(f)=\dim\rmE(f)=\codim\rmH(f^*).
\eneq
In the theorem below,
$$\{t\ge f(x)\}\seteq\set{(x,t)\in\VV\times \R}{t\ge f(x)}$$
is a closed subset of $\VV\times \R$ and
$$\{t\geq -f(x),\;x\in\domo(f)\}
\seteq \set{(x,t)\in\VV\times \R}{t\geq -f(x),\;x\in\domo(f)}$$
is a closed subset of $\domo(f)\times\R$, and hence it is a locally closed
subset of  $\VV\times \R$.

\begin{theorem}\label{th:legendre}
For $f\in\Leg[\VV]$, we have isomorphisms{\rm:}
\eq
&&\EF[\VV](\cor_{\{t\geq f(x)\}})\simeq\cor_{\{t\geq -f^*(-y),-y\in\domo(f^*)\}}\tens\ori_{\rmE(f)}\,[-\rmd(f)],\label{eq:legendre1}\\
&&\EF[\VV](\cor_{\{t\geq -f(x),\;x\in\domo(f)\}})\simeq\cor_{\{t\geq f^*(y)\}}\tens\ori_{\rmH(f)}\,[-\dim\rmH(f)].\label{eq:legendre2}
\eneq
\end{theorem}
\begin{proof}
It is enough to prove the second isomorphism. Equivalently, it is enough to prove:
\eqn
\EF[\VV](\cor_{\{t<-f(x),x\in\domo(f)\}})\simeq\cor_{\{t< f^*(y)\}}\tens\ori_{\rmH(f)}\,[-\dim\rmH(f)].
\eneqn
By the definition
\eqn
\EF[\VV](\cor_{\{t<-f(x),x\in\domo(f)\}})\simeq\reim{p_2}\cor_{\{t<-f(x)+\langle x,y\rangle;x\in\domo(f)\}},
\eneqn
and we have a natural morphism
\eqn
\cor_{\{x\;;\;t<-f(x)+\langle x,y\rangle,\,x\in\domo(f)\}}\to \cor_{\{x\in\rmH(f),\,t<f^*(y)\}}.
\eneqn
This defines the morphism 
\eqn
&&\EF[\VV](\cor_{\{t<-f(x),x\in\domo(f)\}})\to \cor_{\{t< f^*(y)\}}\tens\ori_{\rmH(f)}\,[-\dim\rmH(f)].
\eneqn
To check  that it is an isomorphism, choose $(y,t)\in\VVd\times\R$. Then
\begin{align*}
(\EF[\VV](\cor_{\{t<-f(x),x\in\domo(f)\}}))_{(y,t)}
&\simeq\rsect_c(\VV;\cor_{\{x\in\domo(f)\,;\,t<-f(x)+\langle x,y\rangle\}})\\
&\hspace{-24ex}
\simeq
\begin{cases}
\ori_{\rmH(f)}\,[-\dim\rmH(f)]&\mbox{ if }\set{x\in\domo(f)}{t<-f(x)+\langle x,y\rangle}\not=\emptyset,\\
0&\mbox{ otherwise. }
\end{cases}
\end{align*}
To conclude, notice that $\set{x\in\domo(f)}{t<-f(x)+\langle x,y\rangle}\not=\emptyset$ if and only if $t<f^*(y)$. 
\end{proof}

Recall the notation~\eqref{def:F+}.
\begin{corollary}[{cf.\ \cite{Da12}}]\label{cor:legendre1}
Let $G\subset\VV$ be a non empty compact convex subset. 
Define $p_G^-\cl\VVd\to\R$ by $p_G^-(y)=\inf_{x\in G}\langle x,y\rangle$. Then
\eqn
&&\EF[\VV](\eps_M(\cor_{G}))\simeq \cor_{\{t\geq p^-_G(y)\}}\tens\ori_\VV.
\eneqn
\end{corollary}
\begin{proof}
Define the function
\eqn
&&f(x)=0\mbox{ if $x\in G$ and $f(x)=+\infty$ if $x\notin G$}.
\eneqn
Then $f\in\Leg[\VV]$, $f^*(y)=\sup_{x\in G}\langle x,y\rangle$, $\domo(f^*)=\VVd$, $\rmd(f)=0$
and 
$ -f^*(-y)=p_G^-(y)$. 
Then the result follows from isomorphism~\eqref{eq:legendre1}.
\end{proof}

\begin{corollary}[{cf.\ \cite{Da12}}]\label{cor:legendre2}
Let $U\subset\VV$ be a non empty open convex subset. 
Define $p_U^+\cl\VVd\to\R\sqcup+\infty$ by $p_U^+(y)=\sup_{x\in U}\langle x,y\rangle$. Then
\eqn
&&\EF[\fV]\bl\eps_M(\cor_{U})\br\simeq \cor_{\{t\geq p^+_U(y)\}}\tens\ori_\VV\,[-n].
\eneqn
\end{corollary}
\begin{proof}
As in the proof of Corollary~\ref{cor:legendre1}, define the function
\eqn
&&f(x)=0\mbox{ if $x\in \ol U$ and $f(x)=+\infty$ if $x\notin \ol U$}. 
\eneqn
Then,  $\domo(f)=U$, $\rmH(f)=\VV$ , $f^*=p_U^+$ 
and the result follows from isomorphism~\eqref{eq:legendre2}.
\end{proof}

We endow $\VV$ with an Euclidean norm 
and denote by $\rmd$ the associated distance.  

\begin{lemma} \label{le:lnisconv}
 Let $\phi$ be a strictly decreasing convex function defined on $\R_{>0}$ 
\ro e.g.\ $\phi(t)=-\log t$ or $\phi(t)=t^{-\la}$ with $\la\in\R_{>0}$\rf,
and let $\Omega$ be an open convex subset of $\VV$ such that $\Omega\not=\VV$.
Then the function $\phi\bl\rmd(x,\VV\setminus\Omega)\br$
is a convex function on $\Omega$.
\end{lemma}
\begin{proof}
Let $x_j\in\Omega$ and $a_j\in\R_{>0}$ ($j=1,2$) with $a_1+a_2=1$. Set $x_0=a_1x_1+a_2x_2$ and 
$r_j=\rmd(x_j,\VV\setminus\Omega)$ ($j=1,2$). It is enough to show that
$\phi\bl \rmd(x_0,\VV\setminus\Omega)\br\leq a_1\phi(r_1)+a_2\phi(r_2)$ or equivalently
\eqn
&&\rmd(x_0,\VV\setminus\Omega)\geq 
r_0\eqdot\opb{\phi}\bl a_1\phi(r_1)+a_2\phi(r_2)\br.
\eneqn
This is equivalent to saying that 
\eqn
&&\set{x}{\vert x-x_0\vert<r_0}\subset\Omega.
\eneqn
Since $\phi$ is convex, $\phi(a_1r_1+a_2r_2)\leq a_1\phi(r_1)+a_2\phi(r_2)=\phi(r_0)$. Since $\phi$ is strictly decreasing, we have $a_1r_1+a_2r_2\geq r_0$. Let $x=x_0+y$ with $\vert y\vert<r_0$. Let us show that $x\in\Omega$. 

Since $\vert\frac{r_j}{a_1r_1+a_2r_2}y\vert<r_j$ ($j=1,2$) and $\set{x}{\vert x-x_j\vert<r_j}\subset\Omega$, we have
\eqn
&&x_j+\frac{r_j}{a_1r_1+a_2r_2}y\in\Omega.
\eneqn
Hence $\sum_{j=1}^2 a_j(x_j+\frac{r_j}{a_1r_1+a_2r_2}y)=x_0+y$ belongs to $\Omega$. 
\end{proof}

\section{Laplace transform}

\subsection{Laplace transform}
Recall the $\D_X$-module $\she^{-\varphi}_{U|X}$ and Notation~\ref{not:<?}. 
We saw in \eqref{eq:laplaceEphi} 
\eq
&&\solE_X(\she^{\varphi}_{U|X})\simeq\C_X^\Tam\ctens\C_{\{t=-\Re\phi\}}.\label{eq:laplaceEphiB}
\eneq
We shall apply this result in the following situation.

Let $\WW$ be a complex finite-dimensional vector space of complex dimension $d_\WW$, $\WW^*$ its dual. 
Since $\WW$ is a complex vector space, we shall identify $\ori_\WW$ with $\C$. 
We denote here by 
$\bW$ the projective compactification of $\WW$, we set $\BBH=\bW\setminus\WW$,
and similarly with $\bWd$ and $\BBHd$. We also introduce the bordered spaces
\eqn
&&\fW=(\WW,\bW),\quad \fWd=(\WWd,\bWd).
\eneqn
We set for short
\eqn
&&X=\bW\times\bWd,\, U=\WW\times\WWd,\, Y=X\setminus U.
\eneqn
We shall consider the function 
\eqn
&&\phi\cl\WW\times\WWd\to\C,\quad \phi(x,y)=\langle x,y\rangle.
\eneqn

We introduce the Laplace kernel
\eq\label{eq:Laplaceker}
\shl&\eqdot&\she^{\langle x,y\rangle}_{U|X}.
\eneq
Recall  from Corollary~\ref{cor:fourieradjoint}  
that the kernel of the enhanced Fourier transform   
with respect to the underlying real vector spaces 
of $\WW$ and $\WWd$ is given by 
\eqn
&&L^a_{\WW}\eqdot \cor_{\{t=-\Re\langle x,y\rangle\}}\in\TDC(\iC_{\fW\times\fWd}).
\eneqn

\begin{lemma}\label{le:laplaceD2}
One has the isomorphism in $\TDC(\iC_X)$
\eq\label{eq:Laplaceiso21}
&&\solE_{X}(\shl)\simeq \C^\Tam_X\ctens \Teeim{j}L^a_{\WW},
\eneq
where $j\cl \fW\times\fWd\to X$ is the inclusion.
\end{lemma}
\begin{proof}
This follows immediately from  isomorphism~\eqref{eq:laplaceEphiB}.
\end{proof}
In the sequel, we denote by $\Drm_\WW$ the Weyl algebra on $\WW$. We also 
use the $(\Drm_{\WW\times\WWd},\Drm_{\WWd})$-bimodule $\Drm_{\WW\times\WWd\to\WWd}$ similar to the bimodule $\D_{X\to Y}$ of the theory of 
$\D$-modules, and finally  we denote by 
$\OO_\WW$ the ring of polynomials on $\WW$.

The next result is well-known and goes back  to~\cite{KL85} or before.
\begin{lemma}\label{le:laplaceD1}
There is a natural isomorphism 
\eq\label{eq:Laplaceiso1}
&&\D_\bW(*\BBH)\Dconv\shl\simeq \D_\bWd(*\BBHd)\tens\det\WWd.
\eneq
Here, $\det\WWd=\bigwedge^n\WWd$ where $n=d_{\WWd}$. 
\end{lemma}
\begin{proof}
Using the GAGA principle, we may replace 
$\D_\bW(*\BBH)$ with $\Drm_\WW$, $\D_\bWd(*\BBH)$ with $\Drm_\WWd$, $\shl$ with $\Drm_{\WW\times\WWd} \e^{\langle x,y\rangle}$
and thus 
$\D_\bW(*\BBH)\Dconv\shl$ with 
\eq\label{eq:KatzL}
&&\Drm_{\WWd\from \WW\times\WWd}\ltens[\Drm_{\WW\times\WWd}](\Drm_{\WW\times\WWd} \e^{\langle x,y\rangle}
\ltens[{\OO_{\WW\times\WWd}}]\Drm_{\WW\times\WWd\to\WWd}).
\eneq
This last object is isomorphic to 
\eqn
\bl\Drm_{\WWd\from\WW\times\WWd}\ltens[{\OO_{\WW\times\WWd}}]\Drm_{\WW\times\WWd\to\WW}  \br 
\ltens[\D_{\WW\times\WWd}]\Drm_{\WW\times\WWd} \e^{\langle x,y\rangle}.
\eneqn
Since
\eqn
&&\Drm_{\WWd\from\WW\times\WWd}\ltens[{\OO_{\WW\times\WWd}}]\Drm_{\WW\times\WWd\to\WW} 
\simeq\Drm_{\WW\times\WWd}\tens\det\WWd,
\eneqn
the module~\eqref{eq:KatzL} is isomorphic to $\Drm_{\WW\times\WWd}\e^{\langle x,y\rangle}\tens\det\WWd$.
Finally, one remarks that the natural morphism $\Drm_{\WWd}\to \Drm_{\WW\times\WWd} \e^{\langle x,y\rangle}$ is an isomorphism.
\end{proof}

In the sequel, we shall identify  $\Drm_\WW$ and $\Drm_\WWd$  
 by the correspondence $x_i\leftrightarrow-\partial_{y_i}$, 
$\partial_{x_i}\leftrightarrow y_i$. (Of course, this does not depend on the choice of linear coordinates on $\WW$ and the dual coordinates on $\WWd$.)

\begin{theorem}\label{th:laplace}
 We have an isomorphism in $\Derb((\iDrm_\WW)_\fWd)$
\eq\label{eq:Laplaceiso22}
&&\EF[\WW](\OEn_\fW)\simeq\OEn_\fWd\tens\det\WW\,[-d_\WW].
\eneq
\end{theorem}
\begin{proof}
Set $K=\solE_{\fW\times\fWd}(\shl)$. By Theorem~\ref{th:7412bis}, we have
\eqn
&&\Psi_K^\Tam(\drE_\fW(\shm))\,[-d_\WW]\simeq\drE_\fWd(\shm\Dconv\shl).
\eneqn
By Lemma~\ref{le:laplaceD2}, $K=\C^\Tam_{\fW\times\fWd}\ctens L^a_{\WW}$ 
and by Corollary~\ref{cor:fourieradjoint}, the functor $\EF_\WW$ is isomorphic to the functor 
$\Psi^\Tam_{L^a_\WW}\,[-2d_\WW]$. 
Since 
\eqn
&&\cihom\bl\C^\Tam_{\fW},\drE_\fW(\shm)\br\simeq \drE_\fW(\shm),
\eneqn
we have 
\eqn
&&\Psi_K( \drE_\fW(\shm))\simeq \Psi_{L^a_\VV}( \drE_\fW(\shm))
\simeq \EF_\WW(\drE_\fW(\shm))[2d_\WW].
\eneqn
Therefore
\eqn
&&\EF_\WW(\drE_\fW(\shm))\simeq \drE_\fW(\shm\Dconv\shl)\, [-d_\WW].
\eneqn
 Now  choose $\shm= \D_\bW(*\BBH)$ and  apply Lemma~\ref{le:laplaceD1}.
Since $\drE_\fW(\shm)\simeq\OvE_\fW$ and 
$\drE_\fWd(\shm\Dconv\shl)\simeq\OvE_\fWd\tens\det\WWd$, we obtain 
\eqn
&&\EF_\WW(\OvE_\fW)\simeq\OvE_\fWd\tens\det\WWd\,[-d_\WW].
\eneqn
Hence, it is enough to remark that
\eqn
&&\OvE_\fW\simeq\OEn_\fW\tens\det\WWd\mbox{ and } \OvE_\fWd\simeq\OEn_\fWd\tens\det\WW.
\eneqn
\end{proof}

\begin{remark}
\bnum
\item Symbolically,  isomorphism~\eqref{eq:Laplaceiso22} 
is given by 
$$\OEn_\fWd\tens \det\WW\ni \phi(y)\tens dy
\longmapsto \int \e^{\lan x,y\ran}\phi(y)dy\in \EF[\WW](\OEn_\fW).$$
\item
The identification of $\Drm_\WW$ and $\Drm_\WWd$ 
is given by:
\eqn
&&\Drm_\WW\ni P(x,\partial_x)
\leftrightarrow Q(y,\partial_y)\in\Drm_\WWd\\[1ex]
&&\hs{10ex}\Longleftrightarrow P(x,\partial_x)\e^{\lan x,y\ran}=Q^*(y,\partial_y)\e^{\lan x,y\ran}\\[1ex]
&&\hs{10ex}\Longleftrightarrow P^*(x,\partial_x)\e^{-\lan x,y\ran}=Q(y,\partial_y)\e^{-\lan x,y\ran}.
\eneqn
Here $Q^*(y,\partial_y)$ denotes the formal adjoint operator
of $Q(y,\partial_y)\in\Drm_\WWd$.
\enum
\end{remark}
Applying Corollary~\ref{cor:fourier}, we get:

\begin{corollary}\label{cor:laplace}
Isomorphism~\eqref{eq:Laplaceiso1}  together with the enhanced Fourier-Sato isomorphism induce an isomorphism in $\Derb(\Drm_\WW)$, 
functorial in $F\in\TDC(\iC_{\fW})${\rm:}
\eq\label{eq:Laplaceiso23} 
&&\FHom(F,\OEn_\fW)\simeq \FHom\bl\EF_\WW(F),\OEn_\fWd\br\tens\det\WW\,[-d_\WW].
\eneq
\end{corollary}

As a consequence of Corollary~\ref{cor:laplace}, we recover the main result of~\cite{KS97}: 
\begin{corollary}\label{cor:laplace2}
Isomorphism~\eqref{eq:Laplaceiso1}  together with the  Fourier-Sato isomorphism  induces  an isomorphism in $\Derb(\Drm_\WW)$, 
functorial in  $G\in\Derb_{\R^+}(\C_{\WW})${\rm:} 
\eq\label{eq:Laplaceiso3} 
&&\RHom(G,\Ot[\WW])\simeq \RHom\bl\FS_\WW(G),\Ot[\WWd]\br\tens\det\WW\,[-d_\WW].
\eneq
\end{corollary}
\begin{proof}
By Theorem~\ref{th:fourierefourier} we have the isomorphism
 $\EF[\WW](\epsilon_\WW(G))=\epsilon_\WWd\FS[\WW](G)$, where $\epsilon_\WW$ is given in~\eqref{eq:eMepsilonM}.
Applying  isomorphism~\eqref{eq:Laplaceiso23}  with  $F=\epsilon_\WW(G)$, we obtain
\eqn
&&\FHom(\epsilon_\WW(G),\OEn_\fW)\simeq \FHom\bl\epsilon_\WWd\FS[\WW](G),\OEn_\fWd\br\tens\det\WW\,[-d_\WW].
\eneqn
By Proposition~\ref{pro:fromOEtoEt}, we have 
\eqn
\FHom(\epsilon_\WW(G),\OEn_\fW)&\simeq&\rsect\bl\WW;\alpha_\WW\fihom(\epsilon_\WW(G),\OEn_\fW)\br\\
&\simeq&\rsect\bl\WW;\alpha_\WW\rihom(G,\Ot[\WW])\br\\
&\simeq&\RHom(G,\Ot[\WW]),
\eneqn
and similarly 
$
\FHom(\epsilon_\WWd\FS[\WW](G),\OEn_\fWd)\simeq\RHom(\FS[\WW](G),\Ot[\WWd])$.
\end{proof}

\subsection{Enhanced distributions}
Let $M$ be a real analytic manifold
and consider the natural morphism of bordered spaces
\eqn
&&j\cl M\times\R_\infty \to M\times\PP.
\eneqn
\begin{definition}[{ See~\cite[Def.~8.1.1]{DK13}}]\label{def:DbT}
One sets
\[
\DbT_M = 
\epb j \rhom[\D_\PP](\she_{\C|\BBP}^\tau,\Dbt_{M\times\PP})[1] 
\in \BDC(\iC_{M\times\R_\infty}),
\]
and denotes by $\DbE_M$ the corresponding object of $\TDC(\iD_M)$.
\end{definition}
\begin{proposition}[{ See~\cite[Pro.~8.1.3]{DK13}}]
There are isomorphisms
\eqn
&&\Tr\DbE_M \simeq\DbT_M\quad\text{in $\BDC(\iC_{M\times\R_\infty})$,}\\
&&\C^\Tam_{M}\ctens \DbE_M\simeq \DbE_M \quad\text{in $\TDC(\iD_M)$.}
\eneqn   
\end{proposition}

\begin{remark}
Let  $X$ be a complex manifold and denote as usual by  $X_\R$ the  underlying real analytic manifold to $X$ and $X^c$ the complex conjugate manifold. Then
\eq\label{eq:OTDbT}
&&\OEn_X \simeq \rhom[\opb\pi\D_{X^c}](\opb\pi\O_{X^c},\DbE_{X_\R}).
\eneq 
\end{remark}

\begin{definition}\label{def:classASA}
Let $U\subset M$ be an open subanalytic subset and let $\phi\cl U\to\R$ 
be a continuous map.
We say that $\phi$ is of class (ASA) (almost subanalytic) on $U$ if there exists a $\mathrm{C}^\infty$-function $\psi\cl U\to\R$ such that
\bnum
\item $\psi$ is subanalytic on $M$, that is,
the graph $\Gamma_\psi\subset U\times\R$ is subanalytic in $M\times\bR$,
\item  there exists a constant $C>0$ such that $\vert\phi(x)-\psi(x)\vert\leq C$ for all $x\in U$.
\ee
In such a case, we say that $\psi$ belongs to the (ASA)-class of $\phi$. 
\end{definition}

\Conj
Let $U\subset M$ be an open subanalytic subset and let $\phi\cl U\to\R$ 
be a continuous map, subanalytic on $M$.
We conjecture that such a $\phi$ is always of class (ASA).
\enconj
Let $U$ and $\phi$ be as above with $\phi$ of class (ASA) and let us choose $\psi$ as in Definition~\ref{def:classASA}.
  For an open subanalytic subset $V$ of $M$, we set:
\eq\label{eq:phitempdist0}
&&\e^{-\phi}\Dbt_M(V)=\set{u\in\Db_M(V\cap U)}{\e^\psi u\in\Dbt_M(U\cap V)}.
\eneq
The correspondence $V\mapsto \e^{\phi}\Dbt_M(V)$ defines a sheaf
on $\Msa$, hence an ind-sheaf on $M$. 
We denote this  ind-sheaf  by $\e^{-\phi}\Dbt_M$. 

\begin{theorem}[{ cf.\ \cite[Prop.\;7.3]{Da12}}]\label{th:phitempdist}
Let $\phi$ be a continuous function on a subanalytic open subset $U$ of $M$.
Assume that $\phi$ is of class {\rm(ASA)}.
 Then the right-hand side of~\eqref{eq:phitempdist0} does not depend on the choice of $\psi$ as soon as $\psi$ belongs to the {\rm(ASA)}-class of $\phi$. Moreover, we have isomorphisms in $\BDC(\rm{I}\shd_{M})$
 \eq\label{eq:phitempdist}
&&\ba{l}
 \e^{-\phi}\Dbt_M\simeq\fihom(\C_{\{t\geq\phi(x),x\in U\}},\DbE_M)\\[1ex]
\hs{9ex} \simeq\roim{\pi}\rihom\bl\C_{\{t<\phi(x)\;;\;x\in U\}},\rhom[\D_\BBP]
(\she_{\C|\BBP}^\tau,\Dbt_{M\times\PP})\br.
   \ea\eneq
 In particular, these objects are concentrated in degree $0$.
\end{theorem}
\begin{proof}
(i) Since
\eqn
&&\C^\Tam_M\ctens\C_{W_\phi}\simeq\C^\Tam_M\ctens\C_{W_\psi},
\eneqn
we have 
\eqn
&&\fihom(\C_{\{t\geq\phi(x),x\in U\}},\DbE_M)\simeq \fihom(\C_{\{t\geq\psi(x),x\in U\}},\DbE_M).
\eneqn
Hence, the fact that the right-hand side of~\eqref{eq:phitempdist0} does not depend on the choice of $\psi$ will follow from~\eqref{eq:phitempdist}.
Therefore, we may assume from the beginning that $\phi$ is $\mathrm{C}^\infty$ and subanalytic on $M$ (see Definition~\ref{def:classASA} (i)). 

\vspace{0.3ex}\noindent
(ii) Set $W_\phi=\set{(x,t)}{t<\phi(x);x\in U}$. 
It is enough to show 
\eq
&&\e^{-\phi}\Dbt_M(U)\simeq\RHom(\C_{W_\phi}\,[1],\DbT_M).
\label{eq:pDb}
\eneq

\vspace{0.3ex}\noindent
(iii) Since  $\Dbt_{M\times\PP}(M\times\PP)\to\Dbt_{M\times\PP}(W)$ is surjective, 
 $\partial_t-1$ acting on $\Dbt_{M\times\PP}(W)$ is surjective. It follows that the right-hand side of \eqref{eq:pDb} is concentrated in degree $0$. 
Set for short
 \eq\label{eq:expphidist2}
 &&S(\{t<\phi\})
\eqdot\set{u\in\Db_{M}(U)}%
{\text{$\e^tu\vert_{\{t<\phi(x)\}}$ is tempered on $M\times\PP$}}.
\eneq
Then 
\eqn
&&\Hom(\C_{\{t<\phi(x)\}},\rhom(\she^{ \tau}_{\C|\BBP},\Dbt_{M\times\PP}))\simeq S(\{t<\phi\}).
\eneqn
(iv) Assume that $u\in\sect(U;\e^{-\phi}\Dbt_M)$. Let $w\in\Db(M)$ such that $w_{\vert_V}=\e^\phi u$. 
Then $\e^tu=\e^{t-\phi(x)}w$ and $\e^{t-\phi(x)}$ is a $C^\infty$-function tempered as a distribution on 
the set  $\set{(x,t)}{t<\phi(x)}$. Therefore, $u\in S(\{t<\phi\})$.

\vspace{0.3ex}\noindent
(v) Assume that $u\in S(\{t<\phi\})$.  Let $w\in\Db(M\times\PP)$ such that $w\vert_{\{t<\phi\}}=\e^tu\vert_{\{t<\phi\}}$.
Let us choose a $C^\infty$-function $\chi(t)$ supported on $\set{t\in\R}{\vert t\vert<1}$ with $\int_\R\chi(t+1)\e^t\rmd t=1$. Then
$\chi(t-\phi(x)+1)w=\chi(t-\phi(x)+1)\e^tu$ and 
\eqn
&&\e^\phi u=\int_\R\chi(t-\phi(x)+1)w\,\rmd t\vert_V
\eneqn
is tempered.
\end{proof}

\begin{lemma}\label{le:Dbtdegree0}
Let $U$ be an open subanalytic subset and let $f\cl U\to\R\sqcup\{+\infty\}$ be a map such that
the set $\set{(x,t)\in U\times \R}{t\ge f(x)}$ is  closed in $U\times \R$ 
and is subanalytic in $M\times\bR$.
Then the objects $\fihom(\C_{\{t\geq f(x)\}},\DbE_M)$ and $\FHom(\C_{\{t\geq f(x)\}},\DbE_M)$
are concentrated in degree $0$.
\end{lemma}
\begin{proof}
(i) We know that  $\roim{\pi}\rihom(\C_{\{t\geq f(x)\}},\Dbt_{M\times \bR})$ is concentrated in degree $0$. Then 
$\fihom(\C_{\{t\geq f(x)\}},\DbE_M)$ is represented by the complex 
\eqn\label{eq:Dbtdegree0}
&&\roim{\pi}\rihom(\C_{\{t\geq f(x)\}},\Dbt_{ M\times \bR})\to[\partial_t-1]\roim{\pi}\rihom(\C_{\{t\geq f(x)\}},\Dbt_ {M\times \bR}),
\eneqn
concentrated in degrees $-1$ and $0$. Hence, it is enough to check that the operator $\partial_t-1$ acting on 
$\roim{\pi}\rihom(\C_{\{t\geq f(x)\}},\Dbt_{ M\times \bR})$ is a monomorphism. This follows from
\eqn
&&\set{u\in\Db_M(U\times\R)}{(\partial_t-1)u=0,\;
 \supp(u)\subset\{t\geq f(x)\}}=0.
\eneqn

\vspace{0.3ex}\noindent
(ii) Similarly, the complex $\rsect\bl M;\roim{\pi}\rhom(\C_{\{t\geq f(x)\}},
\Dbt_{ M\times \bR})\br$ is concentrated in degree $0$
and the operator $\partial_t-1$ acting on this object is injective.
\end{proof}

\subsection{Examples}
In this subsection, we shall give some applications of Corollary~\ref{cor:laplace} by using 
Theorems~\ref{th:phitempdist} and~\ref{th:legendre}. As above, $\WW$ is a complex vector space of dimension $d_\WW$. 
 We shall often identify $\WW$ with 
the  underlying real  vector space of real dimension $2d_\WW$. 

\begin{theorem}\label{th:vanishOEn}
 Let $f\in\Leg[\WW]$.
Then
\bnum
\item We have an isomorphism
\eqn
&&\FHom(\C_{\{t\geq f(x)\}}, \OEn_\fW)\,[d_\WW]\\
&&\hs{10ex}\simeq
\FHom(\C_{\{t\geq-f^*(-y),-y\in\domo(f^*)\}},\OEn_\fWd)\tens\;\det \WW
\eneqn
by the Laplace  isomorphism~\eqref{eq:Laplaceiso23}.
\item
Assume further that $f$ is subanalytic on $\bW$ and 
$\domo(f^*)$ is an open subset of $\WWd$.
Then, these objects are concentrated in degree $0$.
\ee
\end{theorem}
\begin{proof}
(i) follows from Corollary~\ref{cor:laplace} and Theorem~\ref{th:legendre}. 

\smallskip\noi
(ii)\  The object $\OEn_\fW$ is represented by the Dolbeault complex
\eqn
&&\OEn_\fW\simeq{\DbT_\bW}^{(0,\scbul)}\eqdot0\to{\DbT_\bW}^{(0,0)}\to[\ol\partial]\cdots\to{\DbT_\bW}^{(0,d_\WW)}\to0,
\eneqn
in which ${\DbT_\WW}^{(0,0)}$ stands in degree $0$.
Then it follows from  Lemma~\ref{le:Dbtdegree0} applied with $M=\bW$
that $\FHom(\C_{\{t\geq f(x)\}},\OEn_\fW)\,[d_\WW]$ is concentrated in degrees $[-d_\WW,0]$
and $\FHom(\C_{\{t\ge -f^*(-y),-y\in\domo(f^*)\}},\OEn_\fWd)$ is concentrated in degrees 
$[0,d_\WW]$. 
\end{proof}

Let $U\subset\WW$ be an open convex subanalytic subset and let $\phi\cl U\to\R$ be a continuous function. We assume that 
$\phi$ is of class (ASA) on $\bW$. 
We define the object $\e^\phi{\Ot[\bW]}\in\Derb(\iC_{\bW})$ by 
 the Dolbeault complex:
\eq\label{eq:dolbeautDbt} 
&&\e^\phi\Ot[\bW]\eqdot0\to \e^\phi{\Dbt_\bW}^{(0,0)}\to[\ol\partial]\cdots\to \e^\phi{\Dbt_\bW}^{(0,d_\WW)}\to0.
\eneq
Hence,  
$\rsect(U;\e^\phi\Ot[\bW])\seteq\RHom(\C_U,\e^\phi\Ot[\bW])$ 
is represented by the complex
\eqn 
&&0\to \e^\phi{\Dbt_\bW}^{(0,0)}(U)\to[\ol\partial]\cdots\to \e^\phi{\Dbt_\bW}^{(0,d_\WW)}(U)\to0.
\eneqn

\begin{corollary}\label{cor:vanishOEn1}
Let $U$ be an open convex and subanalytic subset of $\WW$. Let $\phi\cl U\to\R$ be a continuous function and 
assume that $\phi$ is convex and of class {\rm(ASA)}on $\bW$ \ro see 
{\rm Definition~\ref{def:classASA}}\rp.  Let $\widetilde\phi\in \Leg[\WW]$ be 
the unique function such that $\widetilde\phi\vert_U=\phi$
and  $\domo(\widetilde\phi)=U$.
Then  
\banum
\item
$\rsect(U;\e^\phi\Ot[\bW])$
is concentrated in degree $0$ and its $0$-th cohomology is the space 
$\e^\phi\Ot[\bW](U)\eqdot\set{u\in\sho_\WW(U)}{\text{$\e^{-\phi}u$ 
is tempered in $\bW$}}$.
\item
Assume moreover that ${\widetilde\phi}^*$ is of class {\rm(ASA)} on $\bW$ 
and $\dom(\tw\phi^*)=\WWd$. 
Then  $\rsect(\WWd;\e^{-\tw\phi^*}\Ot[\bWd])$
is concentrated in degree $d_\WW$
and the Laplace tra\-nsform induces an isomorphism between  
the space  $\e^\phi\Ot[\bW](U)$ and 
the space $H^{d_\WW}\bl\rsect(\WWd;\e^{-\tw\phi^*}\Ot[\bWd])\br$.
\eanum
\end{corollary}
\begin{proof}
This follows from Theorems~\ref{th:vanishOEn} and~\ref{th:phitempdist}
applied to the function $\widetilde\phi^*$.
\end{proof}
 For an open subset $\Omega$ of $\WW$, we set for short
\eqn
&&\rmd_\Omega(x)=\rmd(x,\WW\setminus\Omega),
\eneqn
and consider the Dolbeault complex  \eqref{eq:dolbeautDbt} 
in which $\phi(x)=\rmd_\Omega(x)^{-\lambda}$
($\lambda\in\Q_{>0}$).  For a subanalytic open subset $U$ of $\Omega$, we  get the complex 
\eq\label{eq:dolbeautDbt2} 
&&\ba{l}
 0\to \e^{\rmd_\Omega(x)^{-\lambda}}{\Dbt_\bW}^{(0,0)}(U)\To[\ol\partial]\cdots\\[1ex]
\hs{30ex}\cdots\To[\ol\partial] \e^{\rmd_\Omega(x)^{-\lambda}}{\Dbt_\bW}^{(0,d_\WW)}(U)\to0
\ea
\eneq
\begin{corollary}\label{cor:vanishOEn2}
Let $\Omega$ and $U$ be open convex subanalytic subsets of $\fW$
 with $U\subset \Omega$, 
and let $\lambda\in\Q_{>0}$. 
Then, if $\rmd_\Omega(x)^{-\lambda}$ is of class \rm{(ASA)}, the complex ~\eqref{eq:dolbeautDbt2} is concentrated in degree $0$ and its $0$-th cohomology is the space 
\eqn
&&\e^{\rmd^{-\lambda}_\Omega}\Ot[\bW](U)\eqdot\set{u\in\sho_\WW(U)}{\e^{-\rmd_\Omega(x)^{-\lambda}}u\mbox{ is tempered in }\bW}.
\eneqn
\end{corollary}
\begin{proof}
Apply Corollary~\ref{cor:vanishOEn1} with  $ \phi(x)=\rmd_\Omega(x)^{-\lambda}$. 
 This function is convex 
by Lemma~\ref{le:lnisconv}.
\end{proof}

\providecommand{\bysame}{\leavevmode\hbox to2.5em{\hrulefill}\thinspace}

\vspace*{1cm}
\noindent
\parbox[t]{25em}
{\scriptsize{
\noindent
Masaki Kashiwara\\
Research Institute for Mathematical Sciences, Kyoto University \\
Kyoto, 606--8502, Japan\\
and Department of Mathematical Sciences, Seoul National
University\\
599 Gwanak-ro, Gwanak-gu, Seoul 151-747, Korea\\
e-mail:\;masaki@kurims.kyoto-u.ac.jp

\medskip\noindent
Pierre Schapira\\
Institut de Math{\'e}matiques,
Universit{\'e} Pierre et Marie Curie\\
4 Place Jussieu,  75005  Paris\\
and Mathematics Research Unit,
University of  Luxembourg\\
http://http://webusers.imj-prg.fr/\textasciitilde pierre.schapira/\\
e-mail:\;pierre.schapira@imj-prg.fr
}}

\end{document}